\documentclass[11pt,a4paper,reqno]{amsart}
\usepackage[english]{babel}
\usepackage[T1]{fontenc}
\usepackage{verbatim}
\usepackage{palatino}
\usepackage{amsmath}
\usepackage{mathabx}
\usepackage{amssymb}
\usepackage{amsthm}
\usepackage{amsfonts}
\usepackage{graphicx}
\usepackage{esint}
\usepackage{color}
\usepackage{mathtools}
\DeclarePairedDelimiter\ceil{\lceil}{\rceil}
\DeclarePairedDelimiter\floor{\lfloor}{\rfloor}
\usepackage[colorlinks = true, citecolor = black]{hyperref}
\title{Vertical versus horizontal Sobolev spaces}
\author{Katrin F\"assler and Tuomas Orponen}
\address{Department of Mathematics and Statistics \\ University of Jyv\"askyl\"a, P.O. Box. 35 (MaD), FI-40014 University of Jyv\"askyl\"a \\ Finland}
\email{katrin.s.fassler@jyu.fi}
\address{Department of Mathematics and Statistics\\ University of Helsinki,
P.O. Box 68 (Pietari Kalmin katu 5)\\
FI-00014 University of Helsinki\\
Finland}
\email{tuomas.orponen@helsinki.fi}
\date{\today}
\thanks{T.O. is supported by the Academy of Finland via the project \emph{Quantitative rectifiability in Euclidean and non-Euclidean spaces}, grant No. 309365.}

\newcommand{\R}{\mathbb{R}}

\newcommand{\He}{\mathbb{H}}
\newcommand{\N}{\mathbb{N}}

\newcommand{\C}{\mathbb{C}}

\newcommand{\calD}{\mathcal{D}}

\newcommand{\spt}{\operatorname{spt}}

\newcommand{\Rea}{\operatorname{Re}}

\newcommand{\bmo}{\mathrm{BMO}}

\renewcommand{\Im}{\mathrm{Im\,}}

\def\Barint_#1{\mathchoice
          {\mathop{\vrule width 6pt height 3 pt depth -2.5pt
                  \kern -8pt \intop}\nolimits_{#1}}%
          {\mathop{\vrule width 5pt height 3 pt depth -2.6pt
                  \kern -6pt \intop}\nolimits_{#1}}%
          {\mathop{\vrule width 5pt height 3 pt depth -2.6pt
                  \kern -6pt \intop}\nolimits_{#1}}%
          {\mathop{\vrule width 5pt height 3 pt depth -2.6pt
                  \kern -6pt \intop}\nolimits_{#1}}}

\numberwithin{equation}{section}

\theoremstyle{plain}
\newtheorem{thm}[equation]{Theorem}

\newtheorem{lemma}[equation]{Lemma}

\newtheorem{ex}[equation]{Example}
\newtheorem{cor}[equation]{Corollary}
\newtheorem{proposition}[equation]{Proposition}

\theoremstyle{definition}

\newtheorem{definition}[equation]{Definition}

\theoremstyle{remark}
\newtheorem{remark}[equation]{Remark}

\newcommand{\nref}[1]{(\hyperref[#1]{#1})}

\begin{document}

\begin{abstract} Let $\alpha \geq 0$, $1 < p < \infty$, and let $\He^{n}$ be the Heisenberg group. Folland in 1975 showed that if $f \colon \He^{n} \to \R$ is a function in the \emph{horizontal Sobolev space} $S^{p}_{2\alpha}(\He^{n})$, then $\varphi f$ belongs to the Euclidean Sobolev space $S^{p}_{\alpha}(\R^{2n + 1})$ for any test function $\varphi$. In short, $S^{p}_{2\alpha}(\He^{n}) \subset S^{p}_{\alpha,\mathrm{loc}}(\R^{2n + 1})$. We show that the localisation can be omitted if one only cares for Sobolev regularity in the vertical direction: the horizontal Sobolev space $S_{2\alpha}^{p}(\He^{n})$ is continuously contained in the \emph{vertical Sobolev space} $V^{p}_{\alpha}(\He^{n})$.

Our search for the sharper result was motivated by the following two applications. First, combined with a short additional argument, it implies that bounded Lipschitz functions on $\He^{n}$ have a $\tfrac{1}{2}$-order vertical derivative in $\bmo(\He^{n})$. Second, it yields a fractional order generalisation of the (non-endpoint) \emph{vertical versus horizontal Poincar\'e} inequalities of V. Lafforgue and A. Naor.
\end{abstract}

\maketitle

\tableofcontents

\section{Introduction}

What is the relation between horizontal and vertical regularity of real-valued functions defined on the Heisenberg group $\He^{n} \cong (\mathbb{R}^{2n + 1},\cdot,d)$? For precise definitions, see Section \ref{s:prelim}. As a motivating example, consider a bounded Lipschitz function $f \colon \He^{n} \to \R$.
We denote the space of such functions by $W^{1,\infty}(\mathbb{H}^n)$; this notation is justified e.g.\ by \cite[Theorem 8]{MR1738070} and the references therein.
The restriction of $f\in W^{1,\infty}(\mathbb{H}^n)$ to any vertical line is Euclidean $\tfrac{1}{2}$-H\"older continuous. Conversely, any bounded Euclidean $\tfrac{1}{2}$-H\"older function defined on a fixed vertical line is Lipschitz in the metric $d$, and can be extended to  a function in $W^{1,\infty}(\He^{n})$.
Thus, Euclidean $\tfrac{1}{2}$-H\"older continuity on vertical lines is the sharpest "pointwise" conclusion that can be drawn about the regularity in vertical directions from $f\in W^{1,\infty}(\mathbb{H}^n)$. Corollary \ref{t:main} below shows that "on average", every $f \in W^{1,\infty}(\He^{n})$ actually has a little more vertical regularity:
\begin{displaymath} T^{1/2}f \in \bmo(\He^{n}). \end{displaymath}
This cannot be improved to $T^{1/2}f \in L^{\infty}(\He^{n})$; a counterexample is $f(x) = \min\{\|x\|_{\mathbb{H}},1\}$.

A general framework for our results is provided by the fractional order \emph{horizontal Sobolev spaces} $S^{p}_{\alpha}(\He^{n})$, for $1 < p < \infty$ and $\alpha \geq 0$, introduced by Folland \cite{MR0494315} in the 70s, see Definition \ref{d:Sobolev}. For $\alpha \in \N$, these spaces coincide with the standard horizontal Sobolev spaces $W^{\alpha,p}(\He^{n})$ consisting of functions with $L^{p}$ horizontal derivatives of all orders between $0$ and $\alpha$. Folland \cite[Theorem (4.16)]{MR0494315} showed that if $f \in S^{p}_{2\alpha}(\He^{n})$, then $\varphi f \in S^{p}_{\alpha}(\R^{2n + 1})$ for every test function $\varphi \colon \R^{2n + 1} \to \R$. Here $S^{p}_{\alpha}(\R^{2n + 1})$ refers to the standard Euclidean Sobolev space, see Definition \ref{d:SobR}. In other words,
\begin{equation}\label{follandInclusion} S^{p}_{2\alpha}(\He^{n}) \subset S^{p}_{\alpha,\mathrm{loc}}(\R^{2n + 1}). \end{equation}
In fact, Folland proves an analogue of \eqref{follandInclusion} for all Carnot groups.

The need for localisation in \eqref{follandInclusion} follows from formulae such as
\begin{displaymath} \partial_{i} = X_{i} + \frac{y_{i}}{2}[X_{i},Y_{i}], \qquad 1 \leq i \leq n, \end{displaymath}
which express Euclidean partial derivatives as second order horizontal derivatives with \textbf{polynomial} coefficients. In contrast, $\partial_{2n + 1} = T = [X_{i},Y_{i}]$ is a \textbf{constant} coefficient second order horizontal derivative. This suggests that the localisation in the inclusion \eqref{follandInclusion} can be omitted if one is only interested in the regularity of horizontal Sobolev functions in the vertical "$T$" direction. Such a "global" inclusion was critical for the applications we had in mind, namely Corollary \ref{t:main} and inequality \eqref{LNFrac}. In Definition \ref{def:verticalSobolev}, we will define the \emph{vertical Sobolev space} $V_{\alpha}^{p}(\He^{n})$. With this notation in place, our main result is the following:
\begin{thm}\label{main} $S^{p}_{2\alpha}(\He^{n}) \subset V^{p}_{\alpha}(\He^{n})$ for $\alpha \geq 0$ and $1 < p < \infty$. The inclusion map is continuous. \end{thm}
Having established Theorem \ref{main} in Section \ref{sec3}, we study the existence of "pointwise" $t$-derivatives of horizontal Sobolev functions. A natural fractional $t$-derivative of order $\alpha \in (0,1)$ is given by the principal value
\begin{equation}\label{Talpha} T^{\alpha}f(z,t) := \lim_{\varepsilon \to 0} \int_{\R \setminus (-\varepsilon,\varepsilon)} \frac{f(z,t + r) - f(z,t)}{|r|^{1 + \alpha}} \, dr, \qquad (z,t) \in \R^{2n} \times \R. \end{equation}
Relying on a result of Wheeden \cite{MR0240682}, we can conclude the following:
\begin{proposition}\label{main2} Let $0 < \alpha < 1$, $1 < p < \infty$, and $f \in S^{p}_{2\alpha}(\He^{n})$. Then the limit in \eqref{Talpha} exists a.e. and in $L^{p}(\He^{n})$. The function $T^{\alpha}f \in L^{p}(\He^{n})$ is also a distributional derivative in the sense that
\begin{equation}\label{form2} \int_{\He^{n}} \varphi \cdot T^{\alpha}f = \int_{\He^{n}} f \cdot T^{\alpha}\varphi, \qquad \varphi \in \mathcal{D}. \end{equation}
\end{proposition}
Finally, neither result above covers directly the case of bounded Lipschitz functions mentioned earlier, but a small additional argument yields the promised conclusion:
\begin{cor}\label{t:main} Let $f \in W^{1,\infty}(\He^{n})$. Then $T^{1/2}f$ exists a.e. and in the distributional sense \eqref{form2}, and $T^{1/2}f \in \bmo(\He^{n})$ with $\|T^{1/2}f\|_{\bmo} \lesssim \|\nabla_{\He}f\|_{\infty}$. \end{cor}
The proofs of Proposition \ref{main2} and Corollary \ref{t:main} can be found in Section \ref{sec4}. 

\subsection{Connections to previous work} We now briefly discuss a collection of topics closely related to the results explained above.

\subsubsection{Vertical versus horizontal Poincar\'{e} inequalities} The connection between horizontal and vertical regularity of functions on the Heisenberg group has been recently studied by Austin-Naor-Tessera \cite{MR3095705}, Lafforgue-Naor \cite{MR3273443}, and Naor-Young \cite{NY}. They established a number of estimates named \emph{vertical versus horizontal Poincar\'e inequalities}. A general form, taken from \cite[Theorem 2.1]{MR3273443}, reads as follows:
\begin{equation}\label{LN} \left( \int_{\R} \left[ \int_{\He} \left(\frac{|f(z,t + r) - f(z,t)|}{|r|^{1/2}} \right)^{p} \, dz \, dt \right]^{q/p} \, \frac{dr}{|r|} \right)^{1/q} \lesssim_{p,q} \|\nabla_{\He}f\|_{p} \end{equation}
for $f \in \mathcal{D}$, $q \geq 2$, and $1 < p \leq q < \infty$. The same estimate with $p = q = 2$ was essentially contained in \cite{MR3273443}. In \cite[Theorem 35]{NY}, the authors prove that \eqref{LN} remains valid for $p = 1$ and $q = 2$ in $\He^{n}$ for $n \geq 2$; the case $p = 1$ in $\He^{1}$ remains open (see however \cite[Remark 12]{NY}, and the "\emph{Added in proof}" remark at the end of \cite{NY}).

The relationship between \eqref{LN} and Theorem \ref{main} is the following.
\begin{itemize}
\item Theorem \ref{main} implies the following fractional order generalisation of \eqref{LN}:
\begin{equation}\label{LNFrac} \left( \int_{\R} \left[ \int_{\He^{n}} \left(\frac{|f(z,t + r) - f(z,t)|}{|r|^{\alpha}} \right)^{p} \, dz \, dt \right]^{q/p} \, \frac{dr}{|r|} \right)^{1/q} \lesssim_{\alpha,p,q} \|f\|_{p,2\alpha}, \end{equation}
Here $0 < \alpha < 1$, $q \geq 2$, and $1 < p \leq q < \infty$, see Theorem \ref{main_LN_section}. Our technique gives nothing for $p = 1$, and possible analogues of Theorem \ref{main} in this case present an interesting open problem.
\item Theorem \ref{main} is (likely) a little sharper than \eqref{LN}-\eqref{LNFrac} in the cases $p = q \neq 2$.
\end{itemize}
This additional sharpness turns out crucial in the application to Proposition \ref{main2}, so we discuss it further. Fix $p = q \geq 2$, $0 < \alpha < 1$, and let $f \in S^{p}_{2\alpha}(\He^{n})$. For $z \in \R^{2n}$, consider the functions $f_{z} \colon \R \to \R$, defined by $f_{z}(t) = f(z,t)$. Theorem \ref{main} will imply that $f_{z} \in S^{p}_{\alpha}(\R)$ for a.e. $z \in \R^{2n}$. On the other hand, \eqref{LNFrac} implies that
\begin{equation}\label{form1} \int_{\R} \frac{\|t \mapsto f_{z}(t + r) - f_{z}(t)\|_{L^{p}(\R)}^{p}}{|r|^{1 + \alpha p}} \, dr < \infty \qquad \text{for a.e. } z \in \R^{2n}. \end{equation}
So, the question on the relationship between Theorem \ref{main} and \eqref{LN} (in the case $p = q$) boils down to: \emph{what is the connection between the conditions $f_{z} \in S^{p}_{\alpha}(\R)$ and \eqref{form1}?} This question has complete answers, see \cite[V \S 3.5.2]{MR0290095}, and they depend on $p$. The conclusion is that for $p = 2$, the conditions are equivalent (for functions \emph{a priori} in $L^{2}(\R)$), but for $p > 2$, only the implication
\begin{displaymath} f_{z} \in S^{p}_{\alpha}(\R) \quad \Longrightarrow \quad \eqref{form1} \end{displaymath}
holds. Counterexamples against the other implication can be found in \cite[p. 161, \S 6.8]{MR0290095}. As mentioned above, the sharper conclusion $f_{z} \in S^{p}_{\alpha}(\R)$ will be needed in the proof of Proposition \ref{main2}.

\subsubsection{Regularity theory of subelliptic equations} In the study of certain
subelliptic partial differential equations on $\R^{2n + 1}$ of the form
\begin{displaymath}
\sum_{i=1}^{2n} X_i(a_i(\nabla_{\mathbb{H}} f)) =0,
\end{displaymath}
it is natural to consider weak solutions $f \in W^{1,p}(\mathbb{H}^n)$ without \emph{a priori} assumptions on $Tf$. Then, in order to establish regularity for the full gradient $\nabla f$, one has to derive information about the integrability and smoothness of $Tf$. An overview on this topic can be found in Section 1.3 of \cite{MR2531368}. The first results on the H\"older continuity of gradients of solutions to quasi-linear sub-elliptic equations in divergence form on $\mathbb{H}^n$ were obtained by Capogna \cite[Theorem 1.1]{MR1459590}. As explained in \cite[Section 3]{MR1459590}, fractional derivatives in the $T$-direction play an important role in his approach; see in particular \cite[(2.16)-(2.17)]{MR1459590} and the argument starting at \cite[(3.4)]{MR1459590}.


\subsubsection{Parabolic Lipschitz functions} While studying boundary value problems for the heat equation, Lewis and Murray \cite[Chapter III]{MR1323804} introduced the notion of \emph{parabolic Lipschitz functions}. They are certain real-valued functions defined on the parabolic space $(\R^{n},d_{\textup{par}}) = (\R^{n - 1} \times \R, |\cdot| \times \sqrt{|\cdot|})$. Domains in $\R^{n + 1}$ bounded by the graphs of these functions have turned out to be the natural analogue (in the parabolic setting) of Euclidean Lipschitz domains (in the elliptic setting), see for example \cite{MR1418902,MR1996443}. Parabolic Lipschitz functions are, by definition, Lipschitz continuous in the first $(n - 1)$ "horizontal" variables. In the last "vertical" variable, they are required to have a $\tfrac{1}{2}$-order partial derivative in $\bmo(\R^{n},d_{\textup{par}})$. A first motivation for this paper was, in fact, to find out if Lipschitz functions in $\He^{n}$ "automatically" satisfy this last condition, which needs to be \textbf{assumed} in the parabolic world. Corollary \ref{t:main} shows that this is indeed the case.

\section{Preliminaries}\label{s:prelim}

\subsection{The Heisenberg group} Recall that the \emph{Heisenberg group} $\mathbb{H}^n$ is $\mathbb{R}^{2n+1}$ equipped with the group product
\begin{equation}\label{eq:group}
(x,y,t)\cdot(x',y',t'):=\left(x+x',y+y',t+t'+\sum_{i=1}^n\tfrac{1}{2}x_iy_i'-\tfrac{1}{2}y_ix_i'\right),
\end{equation}
where $x=(x_1,\ldots,x_n)$, $y=(x_1,\ldots,y_n) \in \mathbb{R}^n$, and $t\in\mathbb{R}$.
A frame for the left-invariant vector fields on $\mathbb{H}^n$ is given by
\begin{equation}\label{eq:vfd}
X_i= \partial_{x_i} -\tfrac{y_i}{2}\partial_t,\quad Y_i=\partial_{y_i}+\tfrac{x_i}{2}\partial_{t}\quad (i=1,\ldots,n),\quad T=\partial_t.
\end{equation}
{We also use the notation $X_{n+i}:=Y_i$, $(i=1,\ldots,n)$, where convenient.}
The \emph{horizontal gradient} of a function $f:\mathbb{H}^n \to
\mathbb{R}$ is
\begin{equation}\label{eq:horiz_grad}
\nabla_{\mathbb{H}} f := (X_1f,\ldots,X_nf,Y_1f,\ldots, Y_nf),
\end{equation}
if the \emph{horizontal derivatives} $X_i f$ and $Y_i f$ exist in the
distributional sense or in the classical sense pointwise almost
everywhere. We equip $\mathbb{H}^n$ with a left-invariant metric $d$ that
induces the topology of $\mathbb{R}^{2n+1}$ and is homogeneous with
respect to the \emph{Heisenberg dilations}
\begin{displaymath}
\delta_{\lambda}:\mathbb{H}^n \to \mathbb{H}^n,\quad \delta_{\lambda}(z,t):= (\lambda z,\lambda^2 t)
\end{displaymath}
for all $\lambda >0$. All metrics with these properties are bi-Lipschitz equivalent. {For convenience in explicit computations, we use the \emph{Kor\'{a}nyi distance} given by
\begin{displaymath}
d((z,t),(z',t')):= \|(z',t')^{-1} \cdot (z,t)\|_{\mathbb{H}},
\end{displaymath}
with
\begin{displaymath}
\|(\zeta,\tau)\|_{\mathbb{H}}:=\sqrt[4]{|\zeta|^4 + 16\tau^2}, \qquad (\zeta,\tau)\in\mathbb{R}^{2n}\times\mathbb{R}.
\end{displaymath}
}

\subsection{Function spaces}

In this section, we introduce the spaces studied in the paper. Unless otherwise specified, integration over $\mathbb{H}^{n}$ is performed with respect to the Lebesgue measure $\mathcal{L}^{2n + 1}$, which is a left- and right-invariant Haar measure on $\mathbb{H}^{n}$. We write $\int f(x) \, dx$ (or simply $\int f$) instead of $\int f(x) \,d\mathcal{L}^{2n + 1}(x)$.

\begin{definition}[Schwartz functions] The \emph{Schwartz space} $\mathcal{S} = \mathcal{S}(\He^{n})$ consists of the standard Schwartz functions on $\mathbb{R}^{2n + 1}$, see for instance \cite[Definition 1.6.8 and \S 3.1.9]{MR3469687}. The Schwartz functions on $\R$ will also appear, and will be denoted by $\mathcal{S}(\R)$.
\end{definition}

We will need horizontal Sobolev spaces of fractional order in this paper. The definition involves the fractional Laplace operators $(-{\bigtriangleup }_{p})^{\alpha}$ and $(1 - \bigtriangleup_{p})^{\alpha}$, for $\alpha \in \C$ and $1 < p < \infty$, defined initially on
\begin{displaymath} \mathrm{Dom}((-{\bigtriangleup }_{p})^{\alpha}) := \{f \in L^{p}(\He^{n}) : (-\bigtriangleup_{p})^{\alpha}f \in L^{p}(\He^{n})\} \subset L^p(\mathbb{H}^n)\end{displaymath}
and
\begin{equation}\label{OneMinusDelta} \mathrm{Dom}((1 - {\bigtriangleup }_{p})^{\alpha}) := \{f \in L^{p}(\He^{n}) : (1 -\bigtriangleup_{p})^{\alpha}f \in L^{p}(\He^{n})\} \subset L^p(\mathbb{H}^n), \end{equation}
respectively. We will never need, in full generality, the definitions of the operators $-\bigtriangleup_{p}$, $(-{\bigtriangleup }_{p})^{\alpha}$, and $(1 - {\bigtriangleup }_{p})^{\alpha}$, so we will not give them here, lengthy as they are. Some special cases are elaborated on in Section \ref{sec3}, and for more information, we refer to either the original work of Folland \cite[p. 181,186]{MR0494315} (where the notation $\mathcal{J}_p$ and $I+\mathcal{J}_p$ was used), or the monograph \cite{MR3469687}, Section 4.3.1 onwards. We will denote by $\bigtriangleup$ (without subscript) the standard sub-Laplacian $\bigtriangleup = \sum_{j = 1}^{2n} X_{j}^{2}$ on $\He^{n}$.

\begin{definition}[Horizontal Sobolev spaces of fractional order]\label{d:Sobolev} Let $\alpha \geq 0$, $1<p<\infty$.
 The  \emph{horizontal Sobolev space of order $\alpha$} is
  \begin{displaymath} S^{p}_{\alpha}(\mathbb{H}^{n}) := (\mathrm{Dom}((-{\bigtriangleup }_{p})^{\alpha/2}), \|\cdot\|_{p,\alpha}), \end{displaymath}
  where $\|\cdot\|_{p,\alpha}$ is the norm
  $$\|f\|_{p,\alpha}:=\|f\|_p+ \|(-{\bigtriangleup }_{p})^{\alpha/2}(f)\|_{p}.$$
\end{definition}

We briefly discuss some fundamental properties of the spaces $S^{p}_{\alpha}(\He^{n})$; for more information, see \cite[Section 4.4]{MR3469687}.
\begin{remark}\label{sobolevRemark} (a) For $\alpha = k \in \N$ and $1 < p < \infty$, the space $S_{k}^{p}(\He^{n})$ coincides with "standard" horizontal Sobolev space
\begin{displaymath} W^{k,p}(\He^{n}) := \{f \in L^{p}(\He^{n}) : X^{\gamma}f \in L^{p}(\He^{n}) \text{ for all } \gamma \in \{1,\ldots,2n\}^{\ast} \text{ with } |\gamma| \leq k\}. \end{displaymath}
Here $X^{\gamma} = X_{\gamma_{1}}\cdots X_{\gamma_{l}}f$ stands for the distributional horizontal derivative corresponding to the multi-index $\gamma = (\gamma_{1},\ldots,\gamma_{l}) \in \{1,\ldots,2n\}^{\ast}$. Also, the quantity
\begin{displaymath} \|f\|_{W^{k,p}} := \sum_{|\gamma| \leq k} \|X^{\gamma}f\|_{p}, \qquad f \in W^{k,p}(\He^{n}), \end{displaymath}
is equivalent $\|f\|_{p,k}$. For the proof of these statements, see \cite[Corollary (4.13)]{MR0494315}. In this paper, we will only need the special case
\begin{equation}\label{W2Equivalence} \|f\|_{2,1} \sim \|f\|_{2} + \|\nabla_{\He}f\|_{2}, \qquad f \in W^{1,2}(\He^{n}). \end{equation}

(b)  The space $\mathcal{D}$ of smooth and compactly supported functions is dense in $(S^{p}_{\alpha},\|\cdot\|_{p,\alpha})$ for all $\alpha \geq 0$ and $1 < p < \infty$, see \cite[Theorem (4.5)]{MR0494315}.

(c) For $\alpha \geq 0$ and $1 < p < \infty$, the space $S^{p}_{\alpha}(\mathbb{H}^{n})$ coincides with  $\mathrm{Dom}((1-{\bigtriangleup }_{p})^{\alpha/2})$  and the following norms are equivalent to $\|\cdot\|_{p,\alpha}$:
\begin{displaymath}
f\mapsto \|f\|_p+ \|(1-{\bigtriangleup }_{p})^{\alpha/2}(f)\|_{p} \quad\text{and}\quad
f\mapsto \|(1-{\bigtriangleup }_{p})^{\alpha/2}(f)\|_{p}.
\end{displaymath}
For a proof, see \cite[Proposition (4.1)]{MR0494315}.

(d) For $0 < \alpha < 1$ and $1 \leq p < \infty$, the following class is also sometimes referred to (see e.g. \cite[Section 1.2]{MR3732178}) as the horizontal Sobolev functions of order $\alpha$:
\begin{displaymath} \Lambda^{p,p}_{\alpha}(\He^{n}) := \left\{f \in L^{p}(\He^{n}) : \iint_{\He^{n} \times \He^{n}} \frac{|f(x) - f(y)|_{\mathbb{H}}^{p}}{\|y^{-1} \cdot x\|^{(2n + 2) + \alpha p}} \, dx \, dy < \infty \right\}. \end{displaymath}
The relationship between the spaces $S^{p}_{\alpha}(\He^{n})$ and $\Lambda^{p,p}_{\alpha}(\He^{n})$ is described in \cite[Theorem 18 \& 20]{MR558675}: for $\alpha \in (0,1)$, the following inclusions hold:
\begin{displaymath} \Lambda^{p,p}_{\alpha}(\He^{n}) \subset S^{p}_{\alpha}(\He^{n}) \text{ for } 1 < p \leq 2 \quad \text{and} \quad S^{p}_{\alpha}(\He^{n}) \subset  \Lambda^{p,p}_{\alpha}(\He^{n}) \text{ for } 2 \leq p < \infty.  \end{displaymath}
In particular, $\Lambda_{\alpha}^{2,2}(\He^{n}) = S_{\alpha}^{2}(\He^{n})$. For us, the main benefit of using the Sobolev spaces $S^{p}_{\alpha}(\He^{n})$ (over $\Lambda^{p,p}_{\alpha}(\He^{n})$) is that the definition works for all $\alpha \geq 0$, and the spaces coincide with the standard horizontal Sobolev spaces for $\alpha \in \N$.

\end{remark}
Having now defined the horizontal Sobolev spaces, we turn to the definition of \emph{vertical Sobolev spaces}. We first need to recall the definition of fractional order Sobolev spaces in $\R^{n}$. We are not being very efficient here, since the Sobolev spaces in $\He^{n}$ and $\R^{n}$ are both covered by Folland's framework \cite{MR0494315}, and hence we could have given a single definition to cover both cases. However, the Fourier-analytic definition below will be convenient to work with. We also remark that these spaces are sometimes called \emph{Bessel potential spaces}.

\begin{definition}[Sobolev spaces in $\R^{n}$]\label{d:SobR} Let $\alpha > 0$, $1 \leq p \leq \infty$, and let $J_{\alpha} \colon \R^{n} \setminus \{0\} \to \R$ be the \emph{Bessel kernel of index $\alpha$}. Thus,
\begin{displaymath} \widehat{J_{\alpha}}(x) = (1 + 4\pi^{2}|x|^{2})^{-\alpha/2}, \qquad x \in \R^{n}. \end{displaymath}
It turns out, see \cite[Proposition 2, p. 132]{MR0290095}, that $J_{\alpha} \in L^{1}(\R^{n})$ with $\|J_{\alpha}\|_{1} = 1$. Hence $\|f \ast J_{\alpha}\|_{p} \leq \|f\|_{p}$ for any $f \in L^{p}(\R^{n})$. Now, we define
\begin{displaymath} S_{\alpha}^{p}(\R^{n}) := \{g \ast J_{\alpha} : g \in L^{p}(\R^{n})\}. \end{displaymath}
For $\alpha = 0$, we also define $S_{0}^{p}(\R^{n}) := L^{p}(\R^{n})$. For $f = g \ast J_{\alpha} \in S^{p}_{\alpha}(\R^{n})$, $\alpha > 0$, we define the norm $\|f\|_{p,\alpha} := \|g\|_{p}$.
\end{definition}

The double meaning for the notation $\|f\|_{p,\alpha}$ should cause no confusion, since the right interpretation will always be clear from the domain of $f$. For more information on the spaces $S^{p}_{\alpha}(\R^{n})$, see \cite[Section V.3.3]{MR0290095}. 

\begin{definition}[Vertical Sobolev spaces in $\He^{n}$]\label{def:verticalSobolev} Let $\alpha \geq 0$ and $1 \leq p < \infty$. For $f \colon \He^{n} \to \R$ and $z \in \R^{2n}$, define a function $f_{z} \colon \R \to \R$ by $f_{z}(t) := f(z,t)$. We write $f \in V^{p}_{\alpha}(\He^{n})$ if $f_{z} \in S^{p}_{\alpha}(\R)$ for a.e. $z \in \R^{2n}$, and
\begin{displaymath} \|f\|_{V^{p}_{\alpha}} := \left( \int_{\R^{2n}} \|f_{z}\|_{p,\alpha}^{p} \, dz \right)^{1/p} < \infty. \end{displaymath}
\end{definition}

Finally, we define the space $\bmo(\He^{n})$ which appeared in Corollary \ref{t:main}.

\begin{definition}[$\bmo(\He^{n})$] A function $f \in L^{1}_{loc}(\He^{n})$ is of \emph{bounded mean oscillation}, denoted
$f\in \mathrm{BMO}(\mathbb{H}^{n})$, if
\begin{displaymath}
\|f\|_{\bmo} := \sup_{B} \fint_B  |f(x)-f_B|\,dx <\infty,
\end{displaymath}
where the supremum is taken over all balls $B$ in $(\mathbb{H}^{n},d)$ and $f_B:= \fint_B f(x) \, dx$.
\end{definition}

\section{Vertical versus horizontal Sobolev spaces}\label{sec3}

This section contains the proof of Theorem \ref{main}, which we repeat below.
\begin{thm}\label{thm1} Let $1 < p < \infty$, $\alpha \geq 0$, and $f \in S^p_{2\alpha}(\He^n)$. Then $f \in V^{p}_{\alpha}(\He^n)$, and
\begin{equation}\label{form8} \|f\|_{V^{p}_{\alpha}} \lesssim_{\alpha,p} \|f\|_{p,2\alpha}. \end{equation}
\end{thm}

\begin{remark} To explain our strategy for proving Theorem \ref{thm1}, we first outline Folland's argument for the local embedding 
\begin{displaymath} S^{p}_{2\alpha}(\He^{n}) \subset S^{p}_{\alpha,\mathrm{loc}}(\R^{2n + 1}), \end{displaymath}
see \cite[Theorem (4.16)]{MR0494315}. Folland shows that if $\varphi \in \mathcal{D}$, then the map $f \mapsto T_{\varphi}(f) := \varphi f$ extends to a bounded operator between $S^{p}_{2\alpha}(\He^{n})$ and $S^{p}_{\alpha}(\R^{2n + 1})$. This is straightforward for $\alpha \in \N$, and Folland deduces the other cases from an interpolation theorem for linear operators between horizontal Sobolev spaces associated to sub-Laplacians spaces on two, possibly different, Carnot groups \cite[Theorem (4.7)]{MR0494315}; here, the groups are $\He^{n}$ and $\R^{2n + 1}$ equipped with the standard (sub-)Laplacians.

To obtain Theorem \ref{thm1}, we cannot afford to multiply $f$ with a test function. However, as in Folland's case, the continuous inclusion $S^{p}_{2\alpha}(\He^{n}) \subset V^{p}_{\alpha}(\He^{n})$ remains clear for $\alpha \in \N$. The main problem is, then, that $V^{p}_{\alpha}(\He^{n})$ is not a horizontal Sobolev space associated to any sub-Laplacian on $\He^{n}$ or $\R^{2n + 1}$, and therefore Folland's interpolation theorem is not directly applicable. It turns out that (complex) interpolation still works, but we need to write it down from "first principles".
\end{remark}

\subsection{The operators $|T|^{\alpha}$ and $(1 - \bigtriangleup)^{-\alpha}$}\label{s:operator} We begin by introducing certain operators $\Lambda_{\alpha}$, $\Rea \alpha \geq 0$. They are initially defined on Schwartz functions, and have the following form:
\begin{equation}\label{LambdaAlpha} \Lambda_{\alpha}(\varphi) = (1 - \bigtriangleup)^{-\alpha}(|T|^{\alpha}\varphi), \qquad \varphi \in \mathcal{S}, \: \Rea \alpha \geq 0. \end{equation}
So, $\Lambda_{\alpha}$ is the composition of the operators $|T|^{\alpha}$ and $(1 - \bigtriangleup)^{-\alpha}$. Here $|T|^{\alpha}$ is the following (Euclidean) Fourier multiplier:
\begin{displaymath} |T|^{\alpha}\varphi(z,t) = \left[(\xi,\tau) \mapsto (2\pi |\tau|)^{\alpha} \hat{\varphi}(\xi,\tau) \right]^{\check{}}(z,t), \qquad (z,t) \in \R^{2n} \times \R. \end{displaymath}
An equivalent definition would be
\begin{equation}\label{form5} |T|^{\alpha}\varphi(z,t) = [(-\bigtriangleup)^{\alpha/2}\varphi_{z}](t), \qquad (z,t) \in \R^{2n} \times \R, \end{equation}
where $(-\bigtriangleup)^{\alpha/2}$ is the standard fractional Laplacian on $\R$, i.e. the Fourier multiplier with symbol $(2\pi |\tau|)^{\alpha}$.

It is evident from Plancherel's theorem that
\begin{equation}\label{form6} |T|^{\alpha}(\mathcal{S}) \subset L^{2}(\He^n), \qquad \Rea \alpha \geq 0. \end{equation}

We then discuss the operators $(1 - \bigtriangleup)^{-\alpha}$ for $\alpha \in \C$.
First, for $f \in L^{p}(\He^n)$, $1 \leq p \leq \infty$, and $\Rea \alpha > 0$, we define $(1 - \bigtriangleup)^{-\alpha}f$ as the convolution $(1 - \bigtriangleup)^{-\alpha}f := f \ast B_{\alpha}$, where $B_{\alpha} \colon \He^n \to \C$ is the following \emph{Bessel kernel} (see \cite[\S 4.3.4]{MR3469687}):
\begin{displaymath} B_{\alpha}(x) := \frac{1}{\Gamma(\alpha)} \int_{0}^{\infty} s^{\alpha - 1}e^{-s}h_{s}(x) \, ds, \qquad x \in \He^n \setminus \{0\}. \end{displaymath}
Here $(x,s) \mapsto h_{s}(x)$, $(x,s) \in \He^n \times (0,\infty)$, is the heat kernel, see \cite[Theorem (3.1)]{MR0494315} or \cite[Theorem 4.2.7]{MR3469687}. The heat kernel is non-negative on $\He^{n} \times (0,\infty)$, and satisfies $h_{s}(x^{-1}) = h_{s}(x)$ for $x \in \He^{n}$ and $s > 0$. It follows that also 
\begin{equation}\label{BInverse} B_{\alpha}(x^{-1}) = B_{\alpha}(x), \qquad x \in \He^{n} \setminus \{0\}, \: \alpha > 0. \end{equation}
Moreover, $\|h_{s}\|_{1} = 1$ for $s > 0$, and consequently $B_{\alpha} \in L^{1}(\He^n)$ with
\begin{displaymath} \|B_{\alpha}\|_{1} \leq \frac{1}{|\Gamma(\alpha)|} \int_{0}^{\infty} s^{\Rea \alpha - 1}e^{-s}\|h_{s}\|_{1} \, ds = \frac{\Gamma(\Rea \alpha)}{|\Gamma(\alpha)|} < \infty, \qquad \Rea \alpha > 0. \end{displaymath}
Thus, by Young's inequality, see \cite[Proposition (1.10)]{MR0494315}, the convolution $f \ast B_{\alpha}$ is well-defined for $f \in L^{p}(\He^n)$, $1 \leq p \leq \infty$, and
\begin{equation}\label{form31} \|(1 - \bigtriangleup)^{-\alpha}f\|_{p} \leq \|B_{\alpha}\|_{1}\, \|f\|_{p}, \qquad 1 \leq p \leq \infty, \: \Rea \alpha > 0. \end{equation}
If the reader is not familiar with the operators $(1 - \bigtriangleup)^{-\alpha}$, it will appear rather confusing that we also mentioned the operators $(1 - \bigtriangleup_{p})^{-\alpha}$ in \eqref{OneMinusDelta}. For $\Rea \alpha > 0$, these operators coincide for all $1 < p < \infty$: the action of $(1 - \bigtriangleup_{p})^{-\alpha}$ on $L^{p}(\He^{n})$, initially defined as in \cite[\S 4.3.2]{MR3469687}, is in fact given by convolution with the kernel $B_{\alpha} \in L^{1}(\He^{n})$, see \cite[Corollary 4.3.11(ii)]{MR3469687}. For this reason, writing $(1 - \bigtriangleup)^{-\alpha}$ for $\Rea \alpha > 0$ is justified.

For the case $\Rea \alpha = 0$, the "abstract" definition of $(1 - \bigtriangleup_{p})^{-\alpha}$ no longer coincides with convolution by an $L^{1}(\He^n)$ function. This case is discussed extensively in \cite[\S 4.3.3]{MR3469687}. The conclusion relevant here that $(1 - \bigtriangleup_{p})^{-\alpha}$ is, for $\Rea \alpha = 0$ and $1 < p < \infty$, given by convolution with a tempered distribution (independent of $p$) which is smooth outside the origin and satisfies Calder\'on-Zygmund estimates, see \cite[(4.31)]{MR3469687}. So, in brief, $(1 - \bigtriangleup_{p})^{-\alpha}$ is a Calder\'on-Zygmund operator, and hence extends to a bounded operator on $L^{p}(\He^n)$ for $1 < p < \infty$. We will denote by $(1 - \bigtriangleup)^{-\alpha}$ this Calder\'on-Zygmund operator. In the special case $\alpha = 0$, we have, as expected, $(1 - \bigtriangleup)^{0} = \mathrm{Id}$; see \cite[Theorem 4.3.6, 1.(a)]{MR3469687}. Moreover, we record that by \cite[Proposition (4.3)]{MR0494315} (or see \cite[Lemma 4.3.8]{MR3469687}), we have the following estimate for $b > 0$ and $1 < p < \infty$:
\begin{equation}\label{LpBound2} \|(1 - \bigtriangleup)^{-\alpha}f\|_{p} \lesssim_{b,p} e^{\pi|\Im \alpha|}\|f\|_{p}, \qquad 1 < p < \infty, \: 0 \leq \Rea \alpha \leq b. \end{equation}
With these definitions in place, a key feature of the family of $(1 - \bigtriangleup)^{-\alpha}$, $\Rea \alpha \geq 0$, is the following, see \cite[Theorem (3.15)(iv)]{MR0494315}:
\begin{proposition}\label{analyticityProp} $f \in L^{p}(\He^n)$, $1 < p < \infty$, then
\begin{displaymath} \alpha \mapsto (1 - \bigtriangleup)^{-\alpha}f \end{displaymath}
is an analytic $L^{p}$-valued function on $\Rea \alpha > 0$ and a continuous $L^{p}$-valued function on $\Rea \alpha \geq 0$.
\end{proposition}

Finally, for $\Rea \alpha < 0$ and $f \in \mathcal{S}$, we define $(1 - \bigtriangleup)^{-\alpha}f$ as in \cite[\S 3]{MR0494315} or \cite[\S 4.3.2]{MR3469687}; for these parameters $(1 - \bigtriangleup_{p})^{-\alpha}$ does not extend to a bounded operator on $L^{p}(\He^{n})$. However, we will only use the definition for $f \in \mathcal{S}$, and then \nref{i} below justifies the shorthand notation $(1 - \bigtriangleup)^{-\alpha}f$. We record the following facts:
\begin{remark}\label{fracLaplaceFacts}
\begin{itemize}
\item[\phantomsection\label{i}(i)] For all $\alpha \in \C$, the operators $f \mapsto (1 - \bigtriangleup)^{-\alpha}f$ are well-defined on $\mathcal{S}$ and preserve $\mathcal{S}$, see \cite[Corollary 4.3.16]{MR3469687}.
\item[(ii)\phantomsection \label{ii}] If $f \in \mathcal{S}$ and $\alpha \in \C$, then $(1 - \bigtriangleup)^{-\alpha}[(1 - \bigtriangleup)^{\alpha}f] = f$. This follows from \cite[Theorem 4.3.6(1a)]{MR3469687}.
\item[(iii)\phantomsection \label{iii}] Let $\Rea \alpha > 0$, $1 \leq p \leq \infty$, $f \in L^{p}(\He^n)$, and $g \in L^{q}(\He^n)$, where $1/p + 1/q = 1$ (for $p = 1$, $q = \infty$). Then,
\begin{displaymath} \int_{\He^n} [(1 - \bigtriangleup)^{-\alpha}f](x)g(x) \, dx = \int_{\He^n} f(x)[(1 - \bigtriangleup)^{-\alpha}g](x) \, dx. \end{displaymath}
This follows simply by recalling that $(1 - \bigtriangleup)^{-\alpha}$ is given by convolution with $B_{\alpha} \in L^{1}(\He^n)$ for $\Rea \alpha > 0$, and recalling \eqref{BInverse}.
\end{itemize}
\end{remark}

After these considerations, the meaning of the definition stated in \eqref{LambdaAlpha} is clear: for $\Rea \alpha \geq 0$ and $\varphi \in \mathcal{S}$, we recall from \eqref{form6} that $g = |T|^{\alpha}\varphi \in L^{2}(\He^n)$, and then $\Lambda_{\alpha}(\varphi) = (1 - \bigtriangleup)^{-\alpha}g \in L^{2}(\He^n)$. For $\Rea \alpha \geq 0$, we will also consider the "formal adjoint" of $\Lambda_{\alpha}$, namely
\begin{displaymath} \Lambda_{\alpha}^{\ast}\psi := |T|^{\alpha} [(1 - \bigtriangleup)^{-\alpha}\psi], \qquad \psi \in \mathcal{S}. \end{displaymath}
The object on the right is well-defined, and in $L^{2}(\He^n)$, because $(1 - \bigtriangleup)^{-\alpha}\psi \in \mathcal{S}$ (recall \nref{i} above). By the statement that $\Lambda^{\ast}_{\alpha}$ is the formal adjoint of $\Lambda_{\alpha}$, we mean that
\begin{equation}\label{adjoint} \int_{\He^n} (\Lambda_{\alpha}\varphi)\psi = \int_{\He^n} \varphi(\Lambda_{\alpha}^{\ast}\psi), \qquad \varphi,\psi \in \mathcal{S}, \: \Rea \alpha > 0. \end{equation}
This equation (for $\Rea \alpha > 0$) easily follows from \nref{i} and \nref{iii} above, and Plancherel.

\begin{remark} The equation \eqref{adjoint} is an understatement in at least two ways. First, it would also extend to the case $\Rea \alpha = 0$, since $(1 - \bigtriangleup)^{-\alpha}$ is self-adjoint on $L^{2}(\He^n)$, see \cite[Theorem (3.15)(v)]{MR0494315}. Second, using the following formula for the (Euclidean) Fourier transform
\begin{displaymath} \widehat{f \ast g}(\xi,\tau) = \iint e^{-2\pi i(\xi,\tau) \cdot (z,t)} \hat{f}\left(\xi - \tfrac{1}{2}i \tau z, \tau\right)g(z,t) \, dz \, dt, \qquad f,g \in L^{1}(\He^n),\end{displaymath}
cf.\ \cite[(3.13)]{MR1101262},
it would be easy to justify that the operators $|T|^{\alpha}$ and $(1 - \bigtriangleup)^{-\alpha}$ commute on $\mathcal{S}$ at least when $\Rea \alpha > 0$, and hence actually $\Lambda_{\alpha} = \Lambda_{\alpha}^{\ast}$ for $\Rea \alpha > 0$.  \end{remark}

Here is finally the main result of the section:
\begin{thm}\label{mainBoundedness} Let $\alpha \geq 0$, $1 < p < \infty$, and $\varphi \in \mathcal{S}$. Then, $\Lambda_{\alpha}\varphi,\Lambda_{\alpha}^{\ast}\varphi \in L^{p}(\He^n)$, and
\begin{displaymath} \|\Lambda_{\alpha}\varphi\|_{p} + \|\Lambda_{\alpha}^{\ast}\varphi\|_{p} \lesssim_{\alpha,p} \|\varphi\|_{p}. \end{displaymath}
\end{thm}

\begin{remark} Theorem \ref{mainBoundedness} is connected with the work of Strichartz \cite{MR1101262} and M\"uller-Ricci-Stein \cite{MR1312498, MRS} on joint spectral multipliers of $\bigtriangleup$ and $iT$. See for example \cite[Section 6]{MR1312498} or \cite[Theorem 2.3]{MRS}, and in particular \cite[Example 3.3]{MRS}, which states that the operators $\Lambda^{\alpha}$ are bounded on $L^{p}$ for $\Rea \alpha = 0$. To the best of our knowledge, Theorem \ref{mainBoundedness} for $\alpha > 0$ is not explicitly contained in \cite{MR1101262, MR1312498, MRS}. On the other hand, there is little doubt that the techniques in those papers would give an alternative proof. It seemed, however, that Theorem \ref{mainBoundedness} is rather more elementary than the level of the most general results in \cite{MR1101262, MR1312498, MRS}, and therefore it was clearest to give a self-contained argument. \end{remark}

We are mostly interested in the following corollary of Theorem \ref{mainBoundedness}:

\begin{cor}\label{mainCor} Let $\alpha \geq 0$, $1 < p < \infty$, and $\varphi \in \mathcal{S}$. Then,
\begin{equation}\label{form32} \||T|^{\alpha}\varphi\|_{p} \lesssim_{\alpha,p} \|(1 - \bigtriangleup)^{\alpha}\varphi\|_{p}. \end{equation}
\end{cor}
\begin{proof} Fix $0 < \alpha \leq 1$. We simply use the $L^{p}$-boundedness of the operator $\Lambda_{\alpha}^{\ast}$. For $\varphi \in \mathcal{S}$, using \nref{ii} in Remark \ref{fracLaplaceFacts}, we write
\begin{displaymath} |T|^{\alpha}\varphi = |T|^{\alpha}\big( (1 - \bigtriangleup)^{-\alpha}[(1 - \bigtriangleup)^{\alpha}\varphi] \big) = \Lambda^{\ast}_{\alpha}[(1 - \bigtriangleup)^{\alpha}\varphi], \end{displaymath}
recalling also from Remark \ref{fracLaplaceFacts}\nref{i} that $(1 - \bigtriangleup)^{\alpha}\varphi \in \mathcal{S}$, so the expression on the right is well-defined. Now \eqref{form32} follows from Theorem \ref{mainBoundedness}. \end{proof}

Theorem \ref{thm1} is now an easy consequence of the result above:
\begin{proof}[Proof of Theorem \ref{thm1}] Fix $\alpha \geq 0$ and $1 < p < \infty$. We start by establishing the inequality \eqref{form8} for $\varphi \in \mathcal{S}$. First, by \cite[Lemma 2, p. 133]{MR0290095}, we have
\begin{displaymath} \|\varphi_{z}\|_{p,\alpha} \lesssim_{\alpha,p} \|\varphi_{z}\|_{p} + \|(-\bigtriangleup)^{\alpha/2}\varphi_{z}\|_{p}, \qquad z \in \R^{2n}, \end{displaymath}
where $(-\bigtriangleup)^{\alpha/2}$ refers to Fourier multiplication on $\R$ by $(2\pi |\tau|)^{\alpha}$. Consequently, by Fubini's theorem, and recalling that the symbol of $|T|^{\alpha}$ is also $(2\pi |\tau|)^{\alpha}$, we infer that
\begin{displaymath} \|\varphi\|_{V_{\alpha}^{p}}^{p} \lesssim_{\alpha,p} \|\varphi\|_{p}^{p} + \||T|^{\alpha}\varphi\|_{p}^{p} \lesssim \|\varphi\|_{p}^{p} + \|(1 - \bigtriangleup)^{\alpha}\varphi\|_{p}^{p}, \end{displaymath}
where the second inequality follows from Corollary \ref{mainCor}. But, by the Remark \ref{sobolevRemark}(c),
\begin{displaymath} \|\varphi\|_{p}^{p} + \|(1 - \bigtriangleup)^{\alpha}\varphi\|_{p}^{p} \sim_{p} \|\varphi\|_{p,2\alpha}^{p}, \qquad 1 < p < \infty. \end{displaymath}
Thus \eqref{form8} holds for $\varphi \in \mathcal{S}$.

We then consider a general function $f \in S^p_{2\alpha}(\He^n)$. We choose a sequence $\{f^{j}\}_{j \in \N} \subset \mathcal{D}$ such that $\|f^{j} - f\|_{p,2\alpha} \to 0$ as $j \to \infty$ and $\|f_{j}\|_{p,2\alpha} \leq 2\|f\|_{p,2\alpha}$ for all $j \in \N$; this is possible as we discussed in Remark \ref{sobolevRemark}(b). Then $f^{j}_{z} \in \mathcal{S}(\R) \subset S^{p}_{\alpha}(\R)$ for every $j \in \N$ and $z \in \R^{2n}$ fixed, so we may find $g^{j}_{z} \in L^{p}(\R)$ with
\begin{equation}\label{form51} f^{j}_{z} = g^{j}_{z} \ast J_{\alpha}. \end{equation}
If we define $g^{j}(z,t) := g^{j}_{z}(t)$, we find from the $\mathcal{S}$-version of \eqref{form8} that
\begin{displaymath} \|g^{j}\|_{p} =\|f^j\|_{V^p_{\alpha}} \lesssim_{\alpha,p} \|f^{j}\|_{{p,2\alpha}} \lesssim \|f\|_{p,2\alpha}, \qquad j \in \N. \end{displaymath}
Since $1 < p < \infty$, we may therefore assume, after passing to a subsequence, that $g^{j}$ converges weakly to some $g \in L^{p}(\He^n)$.

In this proof only, we consider the \emph{vertical convolution}
\begin{displaymath} \varphi \ast^{\textup{v}} \psi(z,t) := \int_{\R} \varphi(z,t - r)\psi(r) \, dr, \qquad \varphi \in L^{p}(\He^n), \: \psi \in L^{1}(\R). \end{displaymath}
By Fubini's theorem, and Young's inequality on $\R$, we have $\varphi \ast^{\textup{v}} \psi \in L^{p}(\He^n)$ with $\|\varphi \ast^{\textup{v}} \psi\|_{L^{p}(\He^n)} \lesssim \|\varphi\|_{L^{p}(\He^n)}\|\psi\|_{L^{1}(\R)}$. In this notation, \eqref{form51} can be rewritten as
\begin{displaymath} f^{j} = g^{j} \ast^{\textup{v}} J_{\alpha}, \end{displaymath}
and we now claim that $f = g \ast^{\textup{v}} J_{\alpha}$. First, using the convergences
\begin{displaymath}\label{form52} g^{j} \ast^{\textup{v}} J_{\alpha} = f_{j} \to f \text{ in } L^{p}(\He^n) \quad \text{and} \quad g^{j} \rightharpoonup g \text{ in } L^{p}(\He^{n}),  \end{displaymath}
we can write
\begin{equation}\label{form52} \int_{\He^{n}} f\varphi = \lim_{j \to \infty} \int_{\He^{n}} (g^{j} \ast^{\textup{v}} J_{\alpha})\varphi = \lim_{j \to \infty} \int_{\He^{n}} g^{j}(J_{\alpha} \ast^{\textup{v}} \varphi) = \int_{\He^{n}} g(J_{\alpha} \ast^{\textup{v}} \varphi) \end{equation}
for any $\varphi \in \mathcal{D}$, using that $J_{\alpha} \ast^{\textup{v}} \varphi \in L^{q}(\He^{n})$ with $1/p + 1/q = 1$.
Finally, since $g \in L^{p}(\He^n)$, it is easy to justify that
\begin{displaymath} \int_{\He^n} g(\varphi \ast^{\textup{v}} J_{\alpha}) = \int_{\He^n} (g \ast^{\textup{v}} J_{\alpha})\varphi, \qquad \varphi \in \mathcal{D},  \end{displaymath}
and this gives $f = g \ast^{\textup{v}} J_{\alpha}$ a.e. in combination with \eqref{form52}. Since $g_{z} \in L^{p}(\R)$ for a.e. $z \in \R^{2n}$, this immediately gives $f \in V^{p}_{\alpha}(\He^{n})$, as desired. \end{proof}

\subsection{Proof of the main estimate}\label{s:proofOfThm} The main task will be to prove the following estimate for $\alpha \geq 0$ and $1 < p < \infty$:
\begin{equation}\label{form53} \left| \int_{\He^n} (\Lambda_{\alpha}\varphi)(x) \psi(x) \, dx \right| \lesssim_{\alpha,p} \|\varphi\|_{p}\|\psi\|_{q}, \qquad \varphi, \psi \in \mathcal{S}, \: \tfrac{1}{p} + \tfrac{1}{q} = 1. \end{equation}
For $\alpha > 0$, the formal adjointness of the operators $\Lambda_{\alpha}$ and $\Lambda_{\alpha}^{\ast}$, recall \eqref{adjoint}, then implies the same estimate for $\Lambda_{\alpha}^{\ast}$. This shows that both $\Lambda_{\alpha}$ and $\Lambda_{\alpha}^{\ast}$ are bounded on $L^{p}(\He^n)$ for $\Rea \alpha > 0$, and the case $\alpha = 0$ is trivial, as $\Lambda_{\alpha} = \Lambda^{\ast}_{\alpha} = \mathrm{Id}$.

Omitting all (admittely very standard) details, the proof of \eqref{form53} can be condensed in the following three steps.
\begin{itemize}
\item[(a)] $\Lambda_{\alpha}$ is bounded on $L^{p}(\He^n)$ for all $\Rea \alpha = 0$ (and not just $\alpha = 0$),
\item[(b)] If $\alpha = A \in 2\N$, then $\Lambda_{A}$ is (almost) a Calder\'on-Zygmund operator, hence bounded on $L^{p}(\He^n)$. The $L^{p}$-boundedness actually holds for all $\Rea \alpha = A$.
\item[(c)] The operator family $\alpha \mapsto \Lambda_{\alpha}$ is analytic (in a suitable sense) for $\Rea \alpha > 0$ and continuous on $\Rea \alpha \geq 0$, so complex interpolation gives the $L^{p}(\He^n)$ boundedness of $\Lambda_{\alpha}$ for all $0 \leq \alpha \leq A$. Since $A \in 2\N$ is arbitrary, \eqref{form53} follows.
\end{itemize}

We then begin executing the steps (a)-(c) carefully. We will need the following "vertical Hilbert transform":
\begin{lemma} For $\varphi \in \mathcal{S}$, define the operator
\begin{displaymath} H\varphi(z,t) := \left[(\xi,\tau) \mapsto \frac{|\tau|}{\tau} \hat{\varphi}(\xi,\tau) \right]^{\check{}}(z,t). \end{displaymath}
Then $H$ extends to a bounded operator on $L^{p}(\He^n)$, for $1 < p < \infty$.
\end{lemma}

\begin{proof} Note that $(H\varphi)_{z}$ is, up to a constant, the usual Hilbert transform of $\varphi_{z}$ and use Fubini's theorem. \end{proof}

The following lemma states that the operator $|T|^{\alpha}$ maps all sufficiently smooth functions into $L^{p}(\He^n)$:

\begin{lemma}\label{LpLemma} Let $\Rea \alpha \geq 0$ and $f \in \mathcal{S}$. Then,
\begin{equation}\label{form20} \||T|^{\alpha}f\|_{p} \lesssim_{p} (1 + |\Im \alpha|)[\|T^{\floor{\Rea \alpha}}f\|_{p} + \|T^{\ceil{\Rea \alpha}}f\|_{p}], \qquad 1 < p < \infty. \end{equation}
For $p = 2$, the term $|\Im \alpha|$ can be omitted.
\end{lemma}

\begin{proof} We write $m^{\alpha}(\tau) := (2\pi|\tau|)^{\alpha}$ for the symbol of the fractional Laplacian $(-\bigtriangleup)^{\alpha/2}$ on $\R$. Fix $z \in \R^{2n}$, $f \in \mathcal{S}$, $\alpha = \alpha_{1} + i\alpha_{2} \in \C$, and $1 < p < \infty$. We note that if $\beta \in \R$, then $(-\bigtriangleup)^{i\beta}$ is bounded on $L^{p}(\R)$ with operator norm $\|m^{i\beta}\|_{L^{p} \to L^{p}} = \|m^{i\beta}\|_{L^{\infty}(\R)} + \|\tau \mapsto |\tau|(m^{i\beta})'(\tau)\|_{L^{\infty}(\R)} \lesssim_{p} 1 + |\beta|$ by the Marcinkiewicz multiplier theorem. For $p = 2$, the term involving the derivative can be omitted. Consequently,
\begin{displaymath} \|(-\bigtriangleup)^{\alpha/2}f_{z}\|_{p} \lesssim_{p} (1 + |\Im \alpha|)\|(-\bigtriangleup)^{\alpha_{1}/2} f_{z}\|_{p}, \qquad f \in \mathcal{S}, \: z \in \R^{2n}, \end{displaymath}
and hence $\||T|^{\alpha}f\|_{p} \lesssim (1 + |\Im \alpha|)\||T|^{\alpha_{1}}f\|_{p}$ by Fubini's theorem. Further, noting that that $|T|^{k}$ and $T^{k}$ differ by at most the vertical Hilbert transform $H$ for $k \in \N$, we obtain
\begin{equation}\label{form54} \||T|^{\alpha_{1}}f\|_{p} \sim_{p} \||T|^{\alpha_{1} - \floor{\alpha_{1}}}T^{\floor{\alpha_{1}}}f\|_{p}. \end{equation}
To proceed, we temporarily denote by $(1 + |T|)^{\beta}$, $\beta \in \R$, the Fourier multiplication on $\He^n \cong \R^{2n+1}$ by $(1 + 4\pi^{2}|\tau|^{2})^{\beta/2}$. Then,
\begin{displaymath} \||T|^{\beta}\varphi\|_{p} \lesssim_{p,\beta} \|(1 + |T|)^{\beta}\varphi\|_{p}, \qquad \varphi \in \mathcal{S}, \: \beta \geq 0, \: 1 < p < \infty, \end{displaymath}
by Fubini's theorem and \cite[Lemma 2, p. 133]{MR0290095}. Also, the operators $(1 + |T|)^{-\beta}$, $\beta \geq 0$, are bounded on $L^{p}(\He^n)$ (even contractions on $L^{p}(\He^n)$), since $(1 - 4\pi^{2}|\tau|^{2})^{-\beta/2}$ is the Fourier transform of the Bessel kernel $J_{\beta}$, recall Definition \ref{d:SobR}. Thus, the right hand side of \eqref{form54} is further bounded by a constant times
\begin{align*} \|(1 + |T|)^{\alpha_{1} - \floor{\alpha_{1}}}T^{\floor{\alpha_{1}}}f\|_{p} & = \|(1 + |T|)^{\alpha_{1} - \ceil{\alpha_{1}}}(1 + |T|)^{\ceil{\alpha_{1}} - \floor{\alpha_{1}}}T^{\floor{\alpha_{1}}}f\|_{p}\\
& \leq \|(1 + |T|)^{\ceil{\alpha_{1}} - \floor{\alpha_{1}}}T^{\floor{\alpha_{1}}}f\|_{p}\\
& \lesssim \|T^{\floor{\alpha_{1}}}f\|_{p} + \|T^{\ceil{\alpha_{1}}}f\|_{p}. \end{align*}
The last inequality is a special case of \cite[Lemma 3, p. 136]{MR0290095}. The proof is complete.  \end{proof}

The following corollary is immediate:

\begin{cor}\label{cor1} Let $\Rea \alpha = 0$. Then $|T|^{\alpha}$ is bounded on $L^{p}(\mathbb{H}^n)$. In particular, both $\Lambda_{\alpha}$ and $\Lambda_{\alpha}^{\ast}$ are bounded on $L^{p}(\mathbb{H}^n)$, and
\begin{equation}\label{form56} \left| \int_{\He^n} (\Lambda_{\alpha}\varphi)(x) \psi(x) \, dx \right| \lesssim_{p} (1 + |\Im \alpha|)e^{\pi |\Im \alpha|}\|\varphi\|_{p}\|\psi\|_{q}, \qquad \varphi, \psi \in \mathcal{S}, \: \tfrac{1}{p} + \tfrac{1}{q} = 1. \end{equation}
\end{cor}

\begin{proof} Recall that $\Lambda_{\alpha}$ and $\Lambda_{\alpha}^{\ast}$ are compositions of $|T|^{\alpha}$ and $(1 - \bigtriangleup)^{-\alpha}$. Then, combine Lemma \ref{LpLemma} with \eqref{LpBound2}. \end{proof}

Next, we prepare for complex interpolation by establishing the necessary analyticity and continuity properties of the operator-valued map $\alpha \mapsto |T|^{\alpha}$.

\begin{proposition}\label{analyticLemma} For $f \in \mathcal{S}$, the $L^{2}(\mathbb{H}^n)$-valued map $\alpha \mapsto |T|^{\alpha}f$ is analytic on $\Rea \alpha > 0$ and continuous on $\Rea \alpha \geq 0$. \end{proposition}

\begin{proof} As before, write $m^{\alpha}(\tau) := (2\pi|\tau|)^{\alpha}$. We first record that, for $\tau \in \R$ fixed,
\begin{equation}\label{form28} n_{\alpha}(\tau) := \frac{d}{d\alpha} m^{\alpha}(\tau) = (2\pi|\tau|)^{\alpha}\ln 2\pi|\tau| \quad \text{and} \quad \frac{d}{d\alpha}n^{\alpha}(\tau) = (2\pi |\tau|)^{\alpha}(\ln 2\pi|\tau|)^{2} \end{equation}
for $\alpha \in \C$. Then, we consider the function $\rho$ defined by
\begin{displaymath} \rho(z,t) = \left[(\xi,\tau) \mapsto n_{\alpha}(\tau)\hat{f}(\xi,\tau)\right]^{\check{}}(z,t), \qquad (z,t) \in \He^n. \end{displaymath}
Clearly $\rho \in L^{2}(\He^n)$ by Plancherel, and we claim that
\begin{equation}\label{form21} \frac{|T|^{\beta}f - |T|^{\alpha}f}{\beta - \alpha} \stackrel{L^{2}}{\to} \rho, \qquad \Rea \alpha \geq 0, \end{equation}
as $\beta \to \alpha$ inside the half-space $\Rea \beta \geq 0$. To see this, fix any $g \in \mathcal{S}$, let $$k := \ceil{\Rea \alpha}/2 + 1,$$ and consider
\begin{align*} \int_{\He^n} \left[\frac{|T|^{\beta}f - |T|^{\alpha}f}{\beta - \alpha} - \rho \right] g & = \int_{\R^{2n}} \int_{\R} \left[\frac{m_{\beta}(\tau) - m_{\alpha}(\tau)}{\beta - \alpha} - n_{\alpha}(\tau) \right] \widehat{f_{z}}(\tau)\widehat{g_{z}}(\tau) \, d\tau \, dz\\
& = \int_{\R^{2n}} \int_{\mathbb{R}}\eta_{\alpha,\beta}(\tau)(1 + 4\pi^{2}|\tau|^{2})^{k}\widehat{f_{z}}(\tau)\widehat{g_{z}}(\tau) \, d\tau \, dz, \end{align*}
where
\begin{displaymath} \eta_{\alpha,\beta}(\tau) := \frac{m_{\beta}(\tau) - m_{\alpha}(\tau)}{(\beta - \alpha)(1 + 4\pi^{2}|\tau|^{2})^{k}} - \frac{n_{\alpha}(\tau)}{(1 + 4\pi^{2}|\tau|^{2})^{k}}, \qquad \Rea \alpha, \Rea \beta \geq 0, \: \tau \in \R. \end{displaymath}
We will show in a moment that $\|\eta_{\alpha,\beta}\|_{L^{\infty}} \lesssim |\alpha - \beta|$, if $\beta$ is close enough to $\alpha$. Once this has been verified, we can estimate (for such $\beta$)
\begin{align*} \left| \int_{\He^n} \left[\frac{|T|^{\beta}f - |T|^{\alpha}f}{\beta - \alpha} - \rho \right] g \right| & \leq \int_{\R^{2n}} \left| \int_{\R} \eta_{\alpha,\beta}(\tau)(1 + 4\pi^{2}|\tau|^{2})^{k}\widehat{f_{z}}(\tau)\widehat{g_{z}}(\tau) \, d\tau \right| \, dz\\
& \lesssim |\alpha - \beta| \int_{\R^{2n}} \|\tau \mapsto (1 + 4\pi^{2}|\tau|^{2})^{k}\widehat{f}_{z}(\tau)\|_{L^{2}(\R)}\|\widehat{g_{z}}\|_{L^{2}(\R)} \, dz\\
& = |\alpha - \beta| \int_{\R^{2n}} \|(1 - \bigtriangleup_{\R})^{k} f_{z}\|_{L^{2}(\R)} \|g_{z}\|_{L^{2}(\R)} \, dz\\
& \lesssim |\alpha - \beta| \left[\sum_{j = 0}^{2k} \|T^{j} f\|_{2} \right]\|g\|_{2}.   \end{align*}
Taking a supremum over $g \in \mathcal{S}$ with $\|g\|_{2} \leq 1$ gives \eqref{form21} (recalling that $f \in \mathcal{S}$, so any finite sum of derivatives of $f$ is in $L^{2}(\He^n)$).

It remains to show that $\|\eta_{\alpha,\beta}\|_{L^{\infty}(\R)} \lesssim |\alpha - \beta|$ for $\beta$ close enough to $\alpha$. For $\tau \in \R$ fixed, two applications of the mean value theorem give
\begin{equation}\label{form22} \left|\frac{m_{\beta}(\tau) - m_{\alpha}(\tau)}{\beta - \alpha} - n_{\alpha}(\tau)\right| = |n_{\gamma}(\tau) - n_{\alpha}(\tau)| \leq \left| \frac{d}{d \alpha}(\alpha \mapsto n_{\alpha}(\tau))|_{\alpha = \zeta} \right||\alpha - \beta| \end{equation}
for some $\gamma,\zeta \in [\alpha,\beta]$. Now, the $\alpha$-derivative on the right was computed in \eqref{form28}, and we obtain
\begin{displaymath} |\eta_{\alpha,\beta}(\tau)| \leq |\alpha - \beta| \frac{(2\pi|\tau|)^{\Rea \zeta}(\ln 2\pi|\tau|)^{2}}{(1 + 4\pi^{2}|\tau|^{2})^{k}} \lesssim |\alpha - \beta|, \end{displaymath}
assuming that $\beta$ is sufficiently close to $\alpha$ so that $\Rea \zeta \in [\Rea \alpha ,\Rea \beta] \subset [0,2k)$. In particular, $\|\eta_{\alpha,\beta}\|_{L^{\infty}(\R)} \lesssim |\alpha - \beta|$, as desired, and the proof of \eqref{form21} is complete. Both claims of the proposition follow immediately from \eqref{form21}. \end{proof}

Now the desired analyticity and continuity properties of $\alpha \mapsto \Lambda_{\alpha}$ easily follow by combining Proposition \ref{analyticLemma} with Proposition \ref{analyticityProp}.

\begin{lemma}\label{analyticityLemma} For $f \in \mathcal{S}$, the $L^{2}(\mathbb{H}^n)$-valued map $f \mapsto \Lambda_{\alpha}f$ is analytic on $\Rea \alpha > 0$ and continuous on $\Rea \alpha \geq 0$. \end{lemma}

\begin{proof} Namely, for $f \in \mathcal{S}$ fixed and $\Rea \alpha, \Rea \beta > 0$, we write
\begin{displaymath} \frac{\Lambda_{\alpha}f - \Lambda_{\beta}f}{\alpha - \beta} = \frac{(1 - \bigtriangleup)^{-\alpha}|T|^{\alpha}f - (1 - \bigtriangleup)^{-\beta}|T|^{\alpha}f}{\alpha - \beta}g + (1 - \bigtriangleup)^{-\beta} \frac{|T|^{\alpha}f - |T|^{\beta}f}{\alpha - \beta}. \end{displaymath}
Since $|T|^{\alpha}f \in L^{2}(\He^n)$ by Plancherel, the first term converges in $L^{2}(\He^n)$ as $\beta \to \alpha$ by Proposition \ref{analyticityProp}. On the other hand, for $\beta$ contained in any compact subset of $\Rea \beta \geq 0$, the operators $(1 - \bigtriangleup)^{-\beta}$ are equicontinuous on $L^{2}(\He^n)$ by \eqref{LpBound2}, or even \eqref{form31}. Thus, the $L^{2}$-convergence of the second term, as $\beta \to \alpha$, follows from Proposition \ref{analyticLemma}.

The continuity on $\Rea \alpha \geq 0$ follows by a similar argument, writing
\begin{displaymath} \Lambda_{\alpha}f - \Lambda_{\beta}f = [(1 - \bigtriangleup)^{-\alpha}|T|^{\alpha}f - (1 - \bigtriangleup)^{-\beta}|T|^{\alpha}f] + (1 - \bigtriangleup)^{-\beta}[|T|^{\alpha}f - |T|^{\beta}f] \end{displaymath}
The first term tends to zero in $L^{2}(\He^n)$ as $\beta \to \alpha$ by Proposition \ref{analyticityProp}, and and the second term does the same by Proposition \ref{analyticLemma}, and the local $L^{2}$-equicontinuity of the family $\{(1 - \bigtriangleup)^{\beta}\}_{\beta \geq 0}$.  \end{proof}

We are now almost prepared to prove the main estimate \eqref{form53}. We will use the following variant of the Phragm\'en-Lindel\"of principle (see \cite[Lemma 1.3.8]{MR3243734} for the case $A = 1$, and deduce the general case by the conformal change of variables $\alpha \mapsto \alpha/A$):
\begin{lemma}\label{PL} Let $A > 0$, $G \colon \{0 \leq \Rea \alpha \leq A\} \to \C$ be continuous and analytic on $0 < \Rea \alpha < A$. Assume that
\begin{equation}\label{aPriori} \sup_{0 \leq \Rea \alpha \leq A} e^{-b|\Im \alpha|} \log |G(\alpha)| < \infty \end{equation}
for some $0 < b < \pi/A$. Then,
\begin{displaymath} |G(\alpha)| \leq \exp \left\{ \frac{\sin \tfrac{\pi \alpha}{A}}{2A} \int_{\R} \left[ \frac{\log |G(i\beta)|}{\cosh \tfrac{\pi \beta}{A} - \cos \tfrac{\pi \alpha}{A}} + \frac{\log |G(A + i\beta)|}{\cosh \tfrac{\pi \beta}{A} + \cos \tfrac{\pi \alpha}{A}} \right] \, d\beta \right\}, \quad 0 < \alpha < A. \end{displaymath} \end{lemma}

We are now prepared to prove the estimate \eqref{form53}.

\begin{proof}[Proof of \eqref{form53}] Fix $f,g \in \mathcal{S}$ with $\|f\|_{p} = 1 = \|g\|_{q}$. Fix also $A \in 2\N \setminus \{0\}$. We consider the map $G = G_{f,g}$ defined by
\begin{displaymath} G(\alpha) := \int_{\He^n} (\Lambda_{\alpha}f)(x)g(x) \, dx, \qquad 0 \leq \Rea \alpha \leq A. \end{displaymath}
Lemma \ref{analyticityLemma} shows that $G$ is analytic on $0 < \Rea \alpha < A$ and continuous on $0 \leq \Rea \alpha \leq A$. We also have the following \emph{a priori} estimate
\begin{equation}\label{form24} |G(\alpha)| \leq \|\Lambda_{\alpha}f\|_{2}\|g\|_{2} \lesssim e^{\pi |\Im \alpha|}\||T|^{\alpha}f\|_{2}\|g\|_{2} \lesssim e^{\pi |\Im \alpha|}[\|f\|_{2} + \|Tf\|_{2}]\|g\|_{2}, \end{equation}
using Cauchy-Schwarz, then \eqref{LpBound2}, and finally Lemma \ref{LpLemma} for $p = 2$. This certainly implies that $G$ satisfies the double exponential growth bound \eqref{aPriori} in Lemma \ref{PL}. So, it remains to study $G(\alpha)$ for $\Rea \alpha \in \{0,A\}$.

For $\Rea \alpha = 0$, we have already noted in \eqref{form56} that
\begin{equation}\label{form23} |G(\alpha)| \lesssim_{p} (1 + |\Im \alpha|)e^{\pi |\Im \alpha|}\|f\|_{p}\|g\|_{q} = (1 + |\Im \alpha|)e^{\pi|\Im \alpha|}, \end{equation}
where $1/p + 1/q = 1$.

We then need to prove a similar estimate for $\Rea \alpha = A$. Note that $|\tau|^{\alpha} = |\tau|^{\alpha - A}\tau^{A}$ for $\tau \in \R$ by $A \in 2\N$. Using this, and that $B_{\alpha}(x) = B_{\alpha}(x^{-1})$ by \eqref{BInverse}, we may write
\begin{align*} G(\alpha) & = \int_{\He^n} [(|T|^{\alpha}f) \ast B_{\alpha}](x)g(x) \, dx\\
& = \int_{\He^n} |T|^{\alpha}f(x)(g \ast B_{\alpha})(x) \, dx\\
& = \int_{\R^{2n}} \int_{\R} (2\pi|\tau|)^{\alpha}\widehat{f_{z}}(\tau) \widehat{(g \ast B_{\alpha})_{z}}(\tau) \, d\tau \, dz\\
& = \int_{\R^{2n}} \int_{\R} \left[ (2\pi|\tau|)^{\alpha - A} \widehat{f_{z}}(\tau) \right] \left[(2\pi \tau)^{A} \widehat{(g \ast B_{\alpha})_{z}}(\tau) \right] \, d\tau \, dz.    \end{align*}
The functions appearing above are so well integrable (recall in particular that $g \ast B_{\alpha} \in \mathcal{S}$) that the manipulations are easily justified.

To proceed, we note that $\Rea (\alpha - A) = 0$, so $(2\pi|\tau|)^{\alpha - A}$ is one of the multipliers we already encountered during the proof of Lemma \ref{LpLemma}. Consequently, the multiplier with symbol $(2\pi |\tau|)^{\alpha - A}$ defines an $L^{p}$-multiplier for all $1 < p < \infty$ with operator norm $\lesssim_{p} 1 + |\Im \alpha|$. Since moreover $T^{A}(g \ast B_{\alpha}) \in \mathcal{S}$, we may apply Parseval's identity one more time, and then H\"older's inequality, and all of this leads to
\begin{displaymath} |G(\alpha)| \lesssim_{p} (1 + |\Im \alpha|)\|f\|_{p}\|T^{A}(g \ast B_{\alpha})\|_{q} = (1 + |\Im \alpha|)\|T^{A}(g \ast B_{\alpha})\|_{q}. \end{displaymath}
To attain an analogue of \eqref{form23} in the case $\Rea \alpha = A$, it therefore remains to prove that
\begin{equation}\label{form30} \|T^{A}(g \ast B_{\alpha})\|_{q} \lesssim_{A,q} e^{\pi|\Im \alpha|}. \end{equation}
First, as $g \in \mathcal{S}$ and $B_{\alpha} \in L^{1}(\mathbb{H}^n)$, we find that $T^{A}(g \ast B_{\alpha}) = (T^{A}g) \ast B_{\alpha} = (1 - \bigtriangleup)^{-\alpha}(T^{A}g)$. Then, using also the semigroup property $(1 - \bigtriangleup)^{-\alpha}\psi = (1 - \bigtriangleup)^{A - \alpha}(1- \bigtriangleup)^{-A}\psi$ for $\psi \in \mathcal{S}$, see \cite[Theorem (3.15)]{MR0494315}, we find that
\begin{displaymath} \|T^{A}(g \ast B_{\alpha})\|_{q} = \|(1 - \bigtriangleup)^{A - \alpha}(1 - \bigtriangleup)^{-A}(T^{A}g)\|_{q} \lesssim_{q} e^{\pi |\Im \alpha|} \|(1 - \bigtriangleup)^{-A}(T^{A}g)\|_{q}. \end{displaymath}
The last inequality was an application of \eqref{LpBound2}. So, it remains to prove that $(1 - \bigtriangleup)^{-A}(T^{A}g)$ is bounded on $L^{q}$. Since $(1 - \bigtriangleup)^{-A}\varphi = (1 - \bigtriangleup)^{-1}(1 - \bigtriangleup)^{-1}\cdots(1 - \bigtriangleup)^{-1}\varphi$ for $\varphi \in \mathcal{S}$, again by the semigroup property stated in \cite[Theorem (3.15)]{MR0494315}, and the $T$-derivative commutes with convolutions, we find that
\begin{displaymath} (1 - \bigtriangleup)^{-A}(T^{A}g) = [(1 - \bigtriangleup)^{-1}T][(1 - \bigtriangleup)^{-1}T] \cdots [(1 - \bigtriangleup)^{-1}T](g). \end{displaymath}
Thus, also observing that each application of $(1 - \bigtriangleup)^{-1}T$ preserves the Schwartz class, the $L^{q}$-boundedness of $g \mapsto (1 - \bigtriangleup)^{-A}(T^{A}g)$ will follow from the $L^{q}$-boundedness of $g \mapsto (1 - \bigtriangleup)^{-1}(Tg) = (Tg) \ast B_{1}$. To this end, we will simply verify that $g \mapsto (Tg) \ast B_{1}$ is a Calder\'on-Zygmund operator on $\He^n$ in the sense described in \cite[XII.5.2]{stein1993harmonic}.

Let $L$ be the distribution
\begin{displaymath} L(\varphi) := -\int_{\He^n} B_{1}(x) (T\varphi)(x) \, dx, \qquad \varphi \in \mathcal{S}. \end{displaymath}
Then, one easily checks that $[(Tg) \ast B_{1}](x) = (g \ast L)(x)$ for $g \in \mathcal{S}$ and $x \in \He^{n}$ (recalling that convolution with a distribution is defined as $(\varphi \ast L)(x) = L(\check{\varphi} \circ \tau_{x^{-1}})$, where $\tau_{x^{-1}}$ is right translation by $x^{-1}$, and $\check{\varphi}(y) = \varphi(y^{-1})$). Also, $L$ clearly coincides on $\He^n \setminus \{0\}$ with the smooth function $TB_{1}$. If $X^{\gamma}$ stands for any horizontal derivative of order $|\gamma| \in \N \cup \{0\}$ (notably $T = X_i X_{n+i} - X_{n+i}X_i$), we have
\begin{equation}\label{form55} |X^{\gamma}B_{1}(x)| \leq \int_{0}^{\infty} |X^{\gamma}h_{s}(x)| \, ds \lesssim_{\gamma} \|x\|_{\mathbb{H}}^{-2 n- |\gamma|}, \qquad x \in \mathbb{H}^n \setminus \{0\},\end{equation}
by \cite[Lemma 4.3.8]{MR3469687}, so in particular
\begin{displaymath} |X^{\gamma}TB_{1}(x)| \lesssim_{\gamma} \|x\|_{\mathbb{H}}^{-(2n + 2) - |\gamma|}, \qquad x \in \He^n \setminus \{0\}. \end{displaymath}
This gives the correct "size and smoothness" assumptions required from the distribution $L$ in \cite[(80)-(82), p.562]{stein1993harmonic}.\footnote{We note that \cite[(81)]{stein1993harmonic} imposes a decay assumption on the Euclidean partial derivatives of $K := TB_{1}$, namely $\partial_{i}K$ for $1 \leq i \leq 2n$, but these can be bounded by the horizontal and vertical derivatives: $|\partial_{i}K(x)| \lesssim |X_{i}K(x)| + \|x\|_{\He}|TK(x)| \lesssim \|x\|^{-(2n + 2) - 1}$ for $x \in \He^{n} \setminus \{0\}$.} We then verify the cancellation condition \cite[(83)]{stein1993harmonic}: if $\Phi$ is a smooth function supported on $B(0,1)$, $R > 0$, and $\Phi^{R} \in \mathcal{D}(\mathbb{H}^n)$ is the function defined by $\Phi^{R}(x) = \Phi(\delta_{1/R}(x))$, we have the uniform bound
\begin{align*} |L(\Phi^{R})| & = \int_{\He^n} B_{1}(x) (T\Phi^{R})(x) = \frac{1}{R^{2}} \int_{\He^n} B_{1}(x)(T\Phi)(\delta_{1/R}(x)) \, dx\\
& = R^{2n} \int_{\He^n} B_{1}(\delta_{R}(x))(T\Phi)(x) \, dx\\
& \stackrel{\eqref{form55}}{\lesssim} \int_{\He^n} \|x\|_{\mathbb{H}}^{-2n} (T\Phi)(x) \, dx \lesssim \|T\Phi\|_{\infty}, \qquad R > 0. \end{align*}
This gives \cite[(83)]{stein1993harmonic}. Therefore, by \cite[Corollary 5.2.4, p. 567]{stein1993harmonic}, the convolution $\varphi \mapsto \varphi \ast L$ extends to a bounded operator on $L^{p}(\He^n)$, $1 < p < \infty$, and in particular
\begin{displaymath} \|(1 - \bigtriangleup)^{-1}(Tg)\|_{q} \lesssim_{q} \|g\|_{q} = 1, \end{displaymath}
as desired. We have now established \eqref{form30}, and therefore shown that \eqref{form23} holds for $\Rea \alpha \in \{0,A\}$.

Now \eqref{form53} follows immediately from the Phragm\'en-Lindel\"of principle, Lemma \ref{PL}. Namely,
\begin{displaymath} |G(\alpha)| \leq \exp \left\{ \frac{\sin(\pi \alpha)}{2A} \int_{\R} \left[ \frac{\log C_{A,p}(1 + |\beta|))e^{\pi |\beta|}}{\cosh \tfrac{\pi \beta}{A} - \cos \tfrac{\pi \beta}{A}} + \frac{\log C_{A,p}(1 + |\beta|))e^{\pi |\beta|}}{\cosh \tfrac{\pi \beta}{A} + \cos \tfrac{\pi \beta}{A}} \right] \, d\beta \right\} < \infty  \end{displaymath}
for $0 < \alpha < A$. Together with the boundary cases $\alpha \in \{0,A\}$ treated separately, this implies \eqref{form53} with implicit constant given by the expression on the right hand side above. \end{proof}

\section{Pointwise fractional $T$-derivatives}\label{sec4}

We have now established that horizontal Sobolev functions of parameters $(\alpha,p)$ are vertical Sobolev functions with parameters $(\alpha/2,p)$. In this section, as the first application, we infer that horizontal Sobolev functions therefore have pointwise a.e. defined fractional $T$-derivatives. In short, we prove Proposition \ref{main2} and Corollary \ref{t:main}. We begin with a remark.

\begin{remark} If $1 \leq p \leq \infty$, and $f \in L^{p}(\He^{n})$, then $f_{z} \in L^{p}(\R)$ for almost every $z \in \R^{2}$. In particular, in the $L^{p}$-equivalence class of any $f \in L^{p}(\He^{n})$ there exists a representative $\bar{f}$ such that $\bar{f}_{z} \in L^{p}(\R)$ for \textbf{every} $z \in \R^{2n}$. Whenever we write $f \in L^{p}(\He^{n})$ in this section, we will have such a representative in mind. \end{remark}

We now introduce the notions of fractional $T$-derivatives we are interested in.

\begin{definition}[Fractional $T$-derivatives]\label{fractionalDerivatives} Let $0 < \alpha < 1$, $1 \leq p \leq \infty$, and $f \in L^{p}(\He^{n})$.
\begin{itemize}
\item[(i)\phantomsection \label{iFrac}] We say that the \emph{fractional $T$-derivative of order $\alpha$ exists at $(z,t) \in \He^{n}$} if the sequence
\begin{equation}\label{truncations} T^{\alpha,\varepsilon}f(z,t) := \int_{|r| > \varepsilon} \frac{f(z,t + r) - f(z,t)}{|r|^{1 + \alpha}} \, dr, \qquad \varepsilon > 0, \end{equation}
converges to a finite limit as $\varepsilon > 0$. Note that, for $\varepsilon > 0$ fixed, the integral in \eqref{truncations} is absolutely convergent, since $f_{z} \in L^{p}(\R)$ by the previous remark, and $r \mapsto \chi_{\R \setminus [-\varepsilon,\varepsilon]}(r)|r|^{-1 - \alpha} \in L^{1}(\R) \cap L^{\infty}(\R)$. In this case, we
write
\begin{displaymath} T^{\alpha}f(z,t) := \text{p.v.} \int \frac{f(z,t + r) - f(z,t)}{|r|^{1 + \alpha}} \, dr := \lim_{\varepsilon \to 0} \int_{|r| > \varepsilon} \frac{f(z,t + r) - f(z,t)}{|r|^{1 + \alpha}} \, dr. \end{displaymath}
\item[(ii)\phantomsection \label{iiFrac}] We say that the fractional $T$-derivative of order
$\alpha$ \emph{exists in $L^{p}$} if $\{T^{\alpha,\varepsilon}f\}_{\varepsilon > 0}$ is a Cauchy sequence in $L^{p}(\He^{n})$.

\item[(iii)\phantomsection \label{iiiFrac}] We say that a distribution $\Lambda$ on $\He^{n}$ is a \emph{distributional $T$-derivative of order $\alpha$} of a function $f \in L^{p}(\He)$, $1 \leq p \leq \infty$, if
\begin{equation}\label{distributionalDef} \Lambda(\varphi) = \int_{\He^{n}} f \cdot T^{\alpha}\varphi, \qquad \varphi \in \mathcal{D}. \end{equation}
It is easy to see that $T^{\alpha}\varphi \in L^{1}(\He^{n}) \cap L^{\infty}(\He^{n})$ for $0 < \alpha < 1$ and $\varphi \in \mathcal{D}$, so the integral on the right is absolutely convergent for $f \in L^{p}(\He^{n})$, $1 \leq p \leq \infty$.
\end{itemize}
\end{definition}

We begin by observing that $L^{p}$-existence implies distributional existence.

\begin{proposition}\label{p:fromLpToDist} Let $0<\alpha<1$, $1\leq p \leq \infty$ and $f\in  L^p(\mathbb{H}^{n})$. If $T^{\alpha}f$ exists in $L^p$, then
$T^{\alpha}f$ is a distributional $T$-derivative of order $\alpha$.  \end{proposition}

\begin{proof} Fix $\varphi \in \mathcal{D}$ and let $q$ be the dual exponent of $p$.
Since $T^{\alpha}f$ exists in $L^{p}$, and $\varphi \in L^q(\mathbb{H}^{n})$, we can first write
\begin{displaymath} \int_{\He^{n}}  \varphi \cdot T^{\alpha}f = \lim_{\varepsilon \to 0} \int_{\R^{2}} \left[ \int_{\R} \varphi(z,t) \int_{|r| > \varepsilon} \frac{f(z,t + r) - f(z,t)}{|r|^{1+\alpha}} \, dr \, dt \right] \, dz \end{displaymath}
For $z \in \R^{2n}$ fixed, we define $\varphi_{z} \in \mathcal{D}(\R)$ and $f_{z} \in L^{p}(\R)$ as before. Note that
\begin{align*} \int_{\R} \varphi_{z}(t) \int_{|r| > \varepsilon} \frac{f_{z}(t + r) - f_{z}(t)}{|r|^{1+\alpha}} \, dr \, dt & = \int_{\R} \varphi_{z}(t) \int_{|s - t| > \varepsilon} \frac{f_{z}(s) - f_{z}(t)}{|s - t|^{1+\alpha}} \, ds \, dt\\
& = \int_{\R} f_{z}(s) \int_{|r - t| > \varepsilon} \frac{\varphi_{z}(t) - \varphi_{z}(s)}{|r - t|^{1+\alpha}} \, dt \, ds\\
&\quad + \iint_{|s - t| > \varepsilon} \frac{f_{z}(s)\varphi_{z}(s) - f_{z}(t)\varphi_{z}(t)}{|s - t|^{1+\alpha}} \, dr \, dr\\
& = \int_{\R} f_{z}(s) \int_{|r - t| > \varepsilon} \frac{\varphi_{z}(t) - \varphi_{z}(s)}{|r - t|^{1+\alpha}} \, dt \, ds.  \end{align*}
These manipulations are justified, because the function
\begin{displaymath} (s,t) \mapsto \frac{f_{z}(s)\varphi_{z}(s) - f_{z}(t)\varphi_{z}(t)}{|s - t|^{1+\alpha}} \end{displaymath}
is integrable over the domain $\{(s,t) \in \R \times \R : |s - t| > \varepsilon\}$, for $\varepsilon > 0$, and its integral is clearly zero. Hence,
\begin{displaymath} \int_{\He^{n}} \varphi \cdot T^{\alpha}f = \lim_{\varepsilon \to 0} \int_{\He^{n}} f \cdot T^{\alpha,\varepsilon}\varphi = \int_{\He^{n}} f \cdot T^{\alpha}\varphi. \end{displaymath}
The last equation follows from the easy fact that $T^{\alpha} \varphi$ exists in $L^{q}$ for $\varphi \in \mathcal{D}$; we will give the details in Lemma \ref{l:LpDerivForSmooth}. This completes the proof. \end{proof}

The proposition above means that we may concentrate on pointwise and $L^{p}$ existence in the sequel. We now recall the statement of Proposition \ref{main2}:
\begin{proposition}\label{main2Technical} Let $1 < p < \infty$, $0 < \alpha < 1$, and $f \in S^{p}_{2\alpha}(\He^{n})$. Then $T^{\alpha}f$ exists pointwise a.e. and in $L^{p}$ -- and hence in the distributional sense. \end{proposition}

The proof is an application of the next result of Wheeden \cite[Theorems 1\&3]{MR0240682}, concerning the existence of fractional derivatives of functions in $S^{p}_{\alpha}(\R)$:
\begin{thm}[Wheeden] Let $0 < \alpha < 1$, $1 < p < \infty$, and $f \in S^{p}_{\alpha}(\R)$. Then
\begin{displaymath} \partial^{\alpha}f(t) := \textup{p.v.} \int \frac{f(t + r) - f(t)}{|r|^{1 + \alpha}} \, dr \end{displaymath}
exists both a.e. and in $L^{p}(\R)$. The latter statement means that $\partial^{\alpha}f \in L^{p}$, and the truncations
\begin{displaymath} t \mapsto \partial^{\alpha,\varepsilon}f(t) := \int_{|r| > \varepsilon} \frac{f(t + r) - f(t)}{|r|^{1 + \alpha}} \, dr \end{displaymath}
converge to $\partial^{\alpha}f$ in $L^{p}$ as $\varepsilon \to 0$. Moreover,
\begin{equation}\label{form10} \|\partial^{\alpha,\varepsilon}f\|_{p} \lesssim_{\alpha,p} \|f\|_{p,\alpha}, \qquad \varepsilon > 0. \end{equation}
\end{thm}

We can then prove Proposition \ref{main2Technical}.

\begin{proof}[Proof of Proposition \ref{main2Technical}] Fix $0 < \alpha < 1$, $1 < p < \infty$, and $f \in S^{p}_{2\alpha}(\He^{n})$. By Theorem \ref{main}, it follows that $f \in V^{p}_{\alpha}(\He^{n})$, hence $f_{z} \in S^{p}_{\alpha}(\R)$ for a.e. $z \in \R^{2n}$. By Wheeden's theorem, and the relation,
\begin{equation}\label{form3} T^{\alpha,\varepsilon}f(z,t) = \partial^{\alpha,\varepsilon}f_{z}(t), \qquad \varepsilon > 0, \end{equation}
we immediately obtain the pointwise a.e. existence of $T^{\alpha}f(z,t)$. Next, to show existence of $T^{\alpha}f$ in $L^{p}$, we infer from \eqref{form3}-\eqref{form10} and \eqref{form8} that
\begin{equation}\label{form4} \|T^{\alpha,\varepsilon}f\|_{p}^{p} = \int_{\R^{2n}} \|\partial^{\alpha,\varepsilon}f_{z}\|_{p}^{p} \, dz \lesssim_{\alpha,p} \int_{\R^{2n}} \|f_{z}\|_{p,\alpha}^{p} \, dz = \|f\|_{V^{p}_{\alpha}}^{p} \lesssim_{\alpha,p} \|f\|_{p,2\alpha}^{p}. \end{equation}
Now the existence of $T^{\alpha}f$ in $L^{p}$ follows from a standard argument, using the $L^{p}$-existence of $T^{\alpha}\varphi$ for $\varphi \in \mathcal{D}$ (see Lemma \ref{l:LpDerivForSmooth} below). Indeed, recall from Remark \ref{sobolevRemark}(b) that $\mathcal{D}$ is dense in $S^{p}_{2\alpha}(\He^{n})$. Thus, for arbitrary $\delta > 0$, there exists $\varphi \in \mathcal{D}$ with $\|f - \varphi\|_{p,2\alpha} < \delta$. Then, if $\varepsilon_{1},\varepsilon_{2} > 0$ are so small that $\|T^{\alpha,\varepsilon_{1}}\varphi - T^{\alpha,\varepsilon_{2}}\varphi\|_{p} < \delta$, we have
\begin{align*} \|T^{\alpha,\varepsilon_{1}}f - T^{\alpha,\varepsilon_{2}}f\|_{p} & \leq \|T^{\alpha,\varepsilon_{1}}(f - \varphi)\|_{p} + \|T^{\alpha,\varepsilon_{1}}\varphi - T^{\alpha,\varepsilon_{2}}\varphi\|_{p} + \|T^{\alpha,\varepsilon_{2}}(f - \varphi)\|_{p}\\
& \stackrel{\eqref{form4}}{\lesssim_{\alpha,p}} \|f - \varphi\|_{p,2\alpha} + \|T^{\alpha,\varepsilon_{1}}\varphi - T^{\alpha,\varepsilon_{2}}\varphi\|_{p} < 2\delta. \end{align*}
This means that $\{T^{\alpha,\varepsilon}f\}_{\varepsilon > 0}$ is a Cauchy sequence in $L^{p}(\He^{n})$, as claimed. \end{proof}

It remains to show the existence of $T^{\alpha}\varphi$ in $L^{p}$ for $\varphi \in \mathcal{D}$.

\begin{lemma}\label{l:LpDerivForSmooth} Let $0 < \alpha < 1$ and $1 \leq p \leq \infty$. Then $T^{\alpha}\varphi$ exists in $L^{p}$ for $\varphi \in \mathcal{D}$.
\end{lemma}

\begin{proof} We begin with the case $p = \infty$. Fix $0 < \varepsilon_{1} < \varepsilon_{2} < \infty$, and note that
\begin{displaymath} |T^{\alpha,\varepsilon_{1}}\varphi(z,t) - T^{\alpha,\varepsilon_{2}}\varphi(z,t)| \leq \int_{|r| < \varepsilon_{2}} \frac{|\varphi(z,t + r) - \varphi(z,t)|}{|r|^{1 + \alpha}} \, dr =: I_{\varepsilon_{2}}(z,t). \end{displaymath}
We claim that $\|I_{\varepsilon}\|_{\infty} \lesssim_{\alpha,\varphi} \varepsilon^{1 - \alpha}$ for all $\varepsilon > 0$ small enough, which implies, by the estimate above, that $\{T^{\alpha,\varepsilon}\varphi\}_{\varepsilon > 0}$ is a Cauchy sequence in $L^{\infty}$. To prove this, let $R > 0$ be such that $\spt \varphi \subset [-R,R]^{2n + 1}$. Let $0 < \varepsilon < R$. First, if $(z,t) \in \R^{2n + 1} \setminus [-2R,2R]^{2n + 1}$, it is clear that $I_{\varepsilon}(z,t) = 0$, since $(z,t + r) \notin \spt \varphi$ for $|r| < \varepsilon$. On the other hand, for $(z,t) \in [-2R,2R]^{2n + 1}$, the mean value theorem gives the estimate
\begin{equation}\label{form57} I_{\varepsilon}(z,t) \leq \|T\varphi\|_{L^{\infty}([-3R,3R]^{2n + 1})} \int_{ |r| < \varepsilon} \frac{dr}{|r|^{\alpha}} \lesssim_{\alpha,\varphi} \varepsilon^{1 - \alpha}, \qquad 0 < \varepsilon < R.  \end{equation}
This completes the proof of the case $p = \infty$. The case $p = 1$ follows from this estimate, and the observation that $\spt I_{\varepsilon} \subset [-2R,2R]^{2n + 1}$ for $0 < \varepsilon < R$. Indeed,
\begin{displaymath} \|T^{\alpha,\varepsilon_{1}}\varphi - T^{\alpha,\varepsilon_{2}}\varphi\|_{1} \leq \int_{[-2R,2R]^{2n + 1}} I_{\varepsilon_{2}}(z,t) \, dz \, dt \lesssim_{\alpha,\varphi} \varepsilon_{2}^{1 - \alpha}, \end{displaymath}
which implies that $\{T^{\alpha,\varepsilon}\varphi\}_{\varepsilon > 0}$ is a Cauchy sequence in $L^{1}(\He^{n})$. The cases $1 < p < \infty$ can, finally, be deduced from the convexity of $L^{p}$-norms, or proven directly using the argument above.  \end{proof}

\subsection{The case of bounded Lipschitz functions}\label{WinftySection}

In this section, we consider functions in
\begin{displaymath} W^{1,\infty}(\He^{n}) = \{f \colon \He^{n} \to \R : f \text{ is bounded and Lipschitz}\}. \end{displaymath}
Comparing with Proposition \ref{main2Technical}, and viewing $W^{1,\infty}(\He^{n})$ as the "$p = \infty$" endpoint of the scale $\{S^{p}_{1}(\He^{n})\}_{1 < p < \infty} = \{W^{1,p}(\He^{n})\}_{1 < p < \infty}$, one might guess that $T^{1/2}f \in L^{\infty}(\He^{n})$. This is false: for $f(x) = \min\{\|x\|_{\He},1\}$, the $\tfrac{1}{2}$-derivative $T^{1/2}f$ has a logarithmic singularity at $0$, see Example \ref{ex:unbounded} below for details. A possible explanation is that $W^{1,\infty}(\He^{n})$ is not the "right" endpoint of $\{S^{p}_{1}(\He^{n})\}_{1 < p < \infty}$, and one should rather consider the space $S^{\infty}_{1}(\He^{n})$ discussed in \cite[Definition 4.4.2]{MR3469687} -- and more generally $S^{\infty}_{2\alpha}(\He^{n})$ for $0 < \alpha < 1$. We do not know if
\begin{equation}\label{form58} T^{\alpha}(S_{2\alpha}^{\infty}(\He^{n})) \subset L^{\infty}(\He^{n}). \end{equation}
Corollary \ref{t:main} shows that $T^{1/2}(W^{1,\infty}(\He^{n})) \subset \bmo(\He^{n})$. Here is a precise statement:
\begin{cor}\label{Winfty} Let $f \in W^{1,\infty}(\He)$. Then $T^{1/2}f$ exists a.e. and in the distributional sense, and
\begin{equation}\label{bmoIneq} \|T^{1/2}f\|_{\bmo} \lesssim \textup{Lip}(f). \end{equation}
\end{cor}

\begin{remark} Corollary \ref{Winfty} leaves open whether $T^{1/2}f$ \emph{exists in $\bmo$} in the sense that
\begin{displaymath} \{T^{1/2,\varepsilon}f\}_{\varepsilon > 0} \subset \bmo(\He^{n}) \end{displaymath}
is a Cauchy sequence in the $\bmo$ norm. We do not know if this is true; the proof below would only imply that $\{T^{1/2,\varepsilon}f\}_{\varepsilon > 0}$ is a bounded sequence in $\bmo$. Yet another open question is the generalisation of Corollary \ref{Winfty} for $\alpha \in (0,1) \setminus \{\tfrac{1}{2}\}$. One could conjecture \eqref{form58}, or perhaps the same with $\bmo(\He^{n})$ on the right hand side; it seems that the weaker conclusion would follow from our proof if $S^{\infty}_{2\alpha}(\He^{n}) \subset \mathcal{C}^{0,2\alpha}(\He^{n})$, but we do not know if this is true. On the other hand, our proof does \textbf{not} imply that $T^{\alpha}$ maps bounded functions in $\mathcal{C}^{0,2\alpha}(\He^{n})$ to $\bmo(\He^{n})$; neither do we have a counterexample. \end{remark}

\begin{proof}[Proof of Corollary \ref{Winfty}] We will use without further mention that Lipschitz functions have a pointwise a.e. defined horizontal gradients in $L^{\infty}(\He^{n})$, see \cite[Proposition 6.12]{MR2312336}; this is a genuinely easier statement than the Pansu-Rademacher theorem.

To establish, first, the pointwise existence of $T^{1/2}f$, we fix a ball $B \subset \He^{n}$, and then a function $\psi = \psi_{B} \in \mathcal{D}$ with $\chi_{2B} \leq \psi \leq \chi_{3B}$. For $(z,t) \in B$ and $\varepsilon > 0$ fixed, we write
\begin{equation}\label{form47} T^{1/2,\varepsilon}f(z,t) = T^{1/2,\varepsilon}(\psi f)(z,t) + T^{1/2,\varepsilon}([1 - \psi]f)(z,t). \end{equation}
Here certainly $\psi f \in W^{1,2}(\He^{n}) = S^{2}_{1}(\He^{n})$, so
\begin{equation}\label{form45} T^{1/2}(\psi f)(z,t) \text{ exists for a.e. } (z,t) \in B \end{equation}
by Proposition \ref{main2Technical}; moreover $T^{1/2}(\psi f) \in L^{2}(\He)$, and $T^{1/2,\varepsilon}(\psi f) \to T^{1/2}(\psi f)$ in $L^{2}(\He^{n})$ as $\varepsilon \to 0$. On the other hand, if $(z,t) \in B$ is arbitrary, we note that $([1 - \psi])f(z,t) = 0 = ([1 - \psi]f)(z,t + r)$ for all $0 < |r| < \delta$, where $\delta = \delta_{B} > 0$, since $(1 - \psi)$ vanishes on a neighbourhood of $B$. Thus,
\begin{equation}\label{form48} T^{1/2,\varepsilon}([1 - \psi]f)(z,t) = \int_{|r| \geq \delta} \frac{([1 - \psi]f)(z,t + r) - ([1 - \psi]f)(z,t)}{|r|^{3/2}} \, dr \end{equation}
for all $(z,t) \in B$ and $0 < \varepsilon < \delta$, and in particular the limit, as $\varepsilon \to 0$, exists. We also point out that $T^{1/2}([1 - \psi]f) \in L^{\infty}(B)$. Combined with \eqref{form45}, this establishes the existence of $T^{1/2}f(z,t)$ for a.e. $(z,t) \in B$. But since $B$ was arbitrary (and the value of $T^{1/2}f(z,t)$ is independent of the choice of $B$, when it exists), we have shown that $T^{1/2}f(z,t)$ exists for a.e. $(z,t) \in \He^{n}$. We also remark that the previous argument gives
\begin{equation}\label{form46} \|T^{1/2,\varepsilon}f - T^{1/2}f\|_{L^{2}(B)} \to 0 \end{equation}
for any fixed ball $B \subset \He^{n}$; indeed, performing the decomposition \eqref{form47} relative to $B$, we infer from \eqref{form48} that
\begin{displaymath} \|T^{1/2,\varepsilon}f - T^{1/2}f\|_{L^{2}(B)} = \|T^{1/2,\varepsilon}(\psi f) - T^{1/2}(\psi f)\|_{L^{2}(B)}  \end{displaymath}
for $0 < \varepsilon < \delta$. Therefore, \eqref{form46} follows from the $L^{2}$-convergence $T^{1/2,\varepsilon}(\psi f) \to T^{1/2}(\psi f)$. Of course, a corollary of \eqref{form46} is that $T^{1/2}f \in L^{2}_{\mathrm{loc}}(\He^{n}) \subset L^{1}_{\mathrm{loc}}(\He^{n})$, which is needed to make sense of various integrals below.

We next prove \eqref{bmoIneq}. We need to show that
\begin{displaymath} \inf_{c \in \R} \frac{1}{r^{2n + 2}} \int_{B(x_{0},r)} |T^{1/2}f(x) - c| \, dx \lesssim \textup{Lip}(f), \qquad x_{0} \in \He, \: r > 0. \end{displaymath}
We may evidently assume that $x_{0} = 0$, $f(0) = 0$, and $\textup{Lip}(f) > 0$ (since $T^{1/2}$ annihilates constants). Replacing $f$ by $g = f/\textup{Lip}(f)$ if necessary, we may assume that $\textup{Lip}(f) = 1$. Moreover, one easily checks that
\begin{displaymath} (T^{1/2}f)(\delta_{r}(x)) = T^{1/2}f_{r}(x) \end{displaymath}
almost everywhere, where $f_{r} = \tfrac{1}{r}[f \circ \delta_{r}]$. Observing moreover that $\textup{Lip}(f_{r}) = \textup{Lip}(f) = 1$ for $r > 0$, and
\begin{displaymath} \frac{1}{r^{2n + 2}} \int_{B(0,r)} |T^{1/2}f(x) - c| \, dx = \int_{B(0,1)} |(T^{1/2}f)(\delta_{r}x) - c| \, dx = \int_{B(0,1)} |T^{1/2}f_{r}(x) - c| \, dx, \end{displaymath}
we conclude that it suffices to consider $r = 1$. We then write $B := B(0,1)$.

As before, fix a function $\psi \in \calD$ with $\chi_{2B} \leq \psi \leq \chi_{3B}$, and a constant $c \in \R$. Then,
\begin{displaymath} \int_{B} |T^{1/2}f(x) - c| \, dx \leq \int_{B} |T^{1/2}(\psi f)| \, dx + \int_{B} |T^{1/2}[f(1 - \psi)](x) - c| \, dx =: I_{1} + I_{2}(c). \end{displaymath}
To handle $I_{1}$, we use Cauchy-Schwarz, then \eqref{form8}, and finally \eqref{W2Equivalence}:
\begin{displaymath} I_{1} \lesssim \left(\int_{\He^{n}} |T^{1/2}(\psi f)(x)|^{2} \, dx \right)^{1/2} \stackrel{(\ast)}{\lesssim} \|\psi f\|_{V^{2}_{1/2}} \lesssim \|\psi f\|_{2} + \|\nabla_{\He} (\psi f)\|_{2} \lesssim 1. \end{displaymath}
In the last inequality we also used the assumption $f(0) = 0$. The estimate $(\ast)$ can be seen rigorously in various ways; one is to use the $L^{2}$-existence of $T^{1/2}(\psi f)$ and write the left hand side as the limit of the numbers $\|T^{1/2,\varepsilon}(\psi f)\|_{2}$. Then, one can safely use Fubini's theorem in $\He^{n} = \R^{2n} \times \R$, and finally apply \eqref{form10} to each $(\psi f)_{z}$, $z \in \R^{2n}$.

To estimate $I_{2}(c)$, we first record that
\begin{displaymath} [f(1 - \psi)](z,t + s) = 0, \qquad (z,t) \in B, \: 0 \leq |s| < \delta,  \end{displaymath}
recalling that $\chi_{2B} \leq \psi \leq 1$; here $\delta > 0$ is a constant depending only on the ambient dimension. Therefore,
\begin{displaymath} T^{1/2}[f(1 - \psi)](z,t) = \int_{|s| \geq \delta} \frac{[f(1 - \psi)](z,t + s)}{|s|^{3/2}} \, ds, \end{displaymath}
which suggests the definition
\begin{displaymath} c := \int_{|s| \geq \delta} \frac{[f(1 - \psi)](\bar{0},s)}{|s|^{3/2}} \, ds.  \end{displaymath}
Consequently,
\begin{displaymath} T^{1/2}[f(1 - \psi)](z,t) - c = \int_{|s| \geq \delta} \frac{[f(1 - \psi)](z,t + s) - [f(1 - \psi)](\bar{0},s)}{|s|^{3/2}} \, ds. \end{displaymath}
Finally, recalling that $\textup{Lip}(f) = 1$, we estimate
\begin{displaymath} |[f(1 - \psi)](z,t + s) - [f(1 - \psi)](\bar{0},s)| \leq d((z,t + s),(\bar{0},s)) = \|(z,t)\|_{\He} \leq 1 \end{displaymath}
for $(z,t) \in B$. It follows that
\begin{displaymath} I_{2}(c) \lesssim \int_{B} \int_{|s| \geq \delta} \frac{ds}{|s|^{3/2}} \, dx \sim 1, \end{displaymath}
and the proof of \eqref{bmoIneq} is complete.

The fact that $T^{1/2}f$ is a derivative in the distributional sense follows from \eqref{form46} and the proof of Proposition \ref{p:fromLpToDist}. Namely, fix $\varphi \in \mathcal{D}$. Since the support of $\varphi$ is contained in some ball $B \subset \He^{n}$, and \eqref{form46} holds, one can first write
\begin{displaymath} \int_{\He^{n}} \varphi \cdot T^{1/2}f = \int_{B} \varphi \cdot T^{1/2}f = \lim_{\varepsilon \to 0} \int_{\He^{n}} \varphi \cdot T^{1/2,\varepsilon}f. \end{displaymath}
Then, exchanging the order of integration as in the proof of Proposition \ref{p:fromLpToDist}, one finds that
\begin{displaymath} \int_{\He^{n}} \varphi \cdot T^{1/2}f = \lim_{\varepsilon \to 0} \int_{\He^{n}} f \cdot T^{1/2,\varepsilon}\varphi = \int_{\He^{n}} f \cdot T^{1/2}\varphi. \end{displaymath}
The last equation used $f \in L^{\infty}(\He^{n})$, and that $T^{1/2}\varphi$ exists in $L^{1}(\He^{n})$ by Lemma \ref{l:LpDerivForSmooth}. The proof of Corollary \ref{Winfty} is complete.  \end{proof}

\subsection{A counterexample for $T^{1/2}f \in L^{\infty}(\He^{n})$} We conclude the section by showing that the vertical $\tfrac{1}{2}$-derivative of $f_{1}(x) := \min\{\|x\|_{\He},1\}$ has a singularity at $0$. For convenience, we prove the same for $f(x) = \min\{\|x\|_{\He},3\}$, but the statement for $f_{1}$ could be obtained by precomposing $f$ with a suitable dilation.

\begin{ex} \label{ex:unbounded}
Consider the function $f \in W^{1,\infty}(\He^{n})$ defined by $f(x) = \min\{\|x\|_{\He},3\}$.  For illustrative purposes, we first note that $T^{1/2}f(0,0)$ does not exist. Indeed, we have
\begin{displaymath}
\int_{|r|>\varepsilon}\frac{f(0,r)-f(0,0)}{|r|^{3/2}} \,dr = \int_{|r| > \varepsilon}\frac{\min\{2|r|^{1/2},3\}}{|r|^{3/2}}\;dr \to \infty
\end{displaymath}
as $\varepsilon \to 0$. More is true: $T^{1/2}f$ is not essentially bounded in the domain
\begin{displaymath}
\Omega:=\{(z,t)\in B(0,1):\; t> 0,\; |z|^4<16 t^2\}.
\end{displaymath}
To see this, we will derive the following lower bound
\begin{equation}\label{eq:integral}
\left|\int_{\mathbb{R}} \frac{f(z,t+r)-f(z,t)}{|r|^{3/2}}\;\mathrm{d}r \right| \geq \max\left\{\frac{1}{2}\ln \left(\frac{t + 1}{4t}\right)-10,0\right\}, \qquad (z,t) \in \Omega.
\end{equation}
(The integral in \eqref{eq:integral} converges absolutely, since $s\mapsto f_{z}(t+s)$ is Lipschitz for $t\neq 0$.)

\medskip

To prove \eqref{eq:integral}, we first consider those values of $r$ for which the integrand in \eqref{eq:integral} is negative. Since $(z,t)\in \Omega \subset B(0,3)$, we have $f(z,t) = \|(z,t)\|_{\mathbb{H}}$. If $r\in \mathbb{R}$ is such that $(z,t+r)\notin B(0,3)$, then $f(z,t+r)-f(z,t)\geq 3 -\|(z,t)\|_{\mathbb{H}} \geq 0$, so it suffices to consider those $r$ where $(z,t+r)\in B(0,3)$ and
\begin{displaymath}
f(z,t+r)-f(z,t)= \sqrt[4]{|z|^4 + 16(t+r)^2}-\sqrt[4]{|z|^4 + 16t^2}.
\end{displaymath}
This last expression is negative exactly for $r\in (-2t,0)$ and clearly $(z,t)\in \Omega$ implies that $(z,t+r)\in B(0,3)$ for such $r$.
We first show that
\begin{equation}\label{eq:integral2}
\left|\int_{-2t}^0 \frac{f(z,t+r)-f(z,t)}{|r|^{3/2}}\;\mathrm{d}r \right|=
\left|\int_{-2t}^0 \frac{\|(z,t+r)\|_{\mathbb{H}}-\|(z,t)\|_{\mathbb{H}}}{|r|^{3/2}}\;\mathrm{d}r \right|
 \leq 10
\end{equation}
for $(z,t) \in \Omega$. By the mean value theorem, there exists for every $(z,t)\in \Omega$ and all $r\in \mathbb{R}$ a number $\tau \in [t,t + r]$ such that
\begin{equation}\label{eq:meanValue}
\sqrt[4]{|z|^2 + 16(t+r)^2}-\sqrt[4]{|z|^2 + 16t^2} =  \frac{8\tau r}{\|(z,\tau)\|_{\mathbb{H}}^3}.
\end{equation}
We split the domain of integration in \eqref{eq:integral2} in  two intervals $(-2t,-t/2)$ and $(-t/2,0)$.
In the range $r\in (-t/2,0)$, we apply \eqref{eq:meanValue} as follows:
\begin{align*}
0\geq \int_{-t/2}^{0} \frac{\|(z,t+r)\|_{\mathbb{H}}-\|(z,t)\|_{\mathbb{H}}}{|r|^{3/2}}\;dr &\geq
\frac{ 8t}{\|(z,t/2)\|^3_{\mathbb{H}}}\int_{-t/2}^0 \frac{r}{(-r)^{3/2}}\,dr
= \frac{ - 32 \cdot t^{3/2}}{\|(z,t)\|^3_{\mathbb{H}}}.
\end{align*}
Therefore,
\begin{displaymath}
\left|\int_{-t/2}^{0} \frac{\|(z,t+r)\|_{\mathbb{H}}-\|(z,t)\|_{\mathbb{H}}}{|r|^{3/2}}\;dr \right| \leq 4.
\end{displaymath}
Next, using the $\frac{1}{2}$-H\"older continuity of $s \mapsto \|(z,t+s)\|_{\mathbb{H}}$, we find
\begin{align*}
\left|\int_{-2t}^{-t/2} \frac{\|(z,t+r)\|_{\mathbb{H}}-\|(z,t)\|_{\mathbb{H}}}{|r|^{3/2}}\;dr \right| \leq 2 \int_{-2t}^{-\frac{t}{2}} \frac{|r|^{1/2}}{|r|^{3/2}}\;dr\leq 6.
\end{align*}
Thus we have established \eqref{eq:integral2}.

For $r\notin (-2t,0)$, the integrand in \eqref{eq:integral} is nonnegative. Thus,
\begin{displaymath}
|T^{1/2} f(z,t)| \geq \int_{3t}^{1} \frac{f(z,t+r)-f(z)}{|r|^{3/2}}\;dr  - 10, \qquad (z,t) \in \Omega, \end{displaymath}
by \eqref{eq:integral2}. For $(z,t)\in \Omega$ and $r\in (3t,1)$, we note that
\begin{displaymath}
|z|^4 + 16 (t+r)^2 \leq |z|^4 + 16 t^2 + 24\leq 25,
\end{displaymath}
and hence $(z,t+r) \in B(0,3)$ and $f(z,t+r)= \|(z,t+r)\|_{\mathbb{H}}$. We apply again the mean value theorem to find $\tau \in [t,t+r]$ such that \eqref{eq:meanValue} holds. Since $|z|^4 \leq 16 t^2$ for $(z,t) \in \Omega$, and since $t^2 \leq \tau^2$, it follows that $\|(z,\tau)\|_{\mathbb{H}} \leq 2^{5/4} \tau^{1/2}$ and thus
\begin{displaymath}
\sqrt[4]{|z|^2 + 16(t+r)^2}-\sqrt[4]{|z|^2 + 16t^2} \geq  \frac{r}{2\tau^{1/2}} \geq \frac{r}{2(t+r)^{1/2}}, \qquad {3t < r < 1}.
\end{displaymath}
Hence
\begin{align*}
\int_{3t}^{1} \frac{\|(z,t+r)\|_{\mathbb{H}}-\|(z,t)\|_{\mathbb{H}}}{|r|^{3/2}}\;dr
&\geq \frac{1}{2}\int_{3t}^{1} \frac{1}{(t+r)^{1/2}} \frac{1}{r^{1/2}}\;dr\\
&\geq \frac{1}{2} \int_{3t}^{1} \frac{1}{(t+r)} \;dr =\frac{1}{2} \ln \left(\frac{t + 1}{4t} \right),
\end{align*}
which completes the proof of \eqref{eq:integral}. Since the lower bound in \eqref{eq:integral} is not bounded for $(z,t) \in \Omega$, we see that $T^{1/2}f\notin L^{\infty}(\mathbb{H}^n)$.
\end{ex}

\section{Vertical versus horizontal Poincar\'{e} inequalities}\label{s:HorizVert}

In this section, we prove the (fractional) \emph{vertical versus horizontal Poincar\'{e} inequality} \eqref{LNFrac}. For brevity of notation, we first introduce the following seminorms.

\begin{definition} Let $1 \leq p, q < \infty$, and let $0<\alpha <1$. For $\varphi \in \mathcal{S}$, we define
\begin{equation}\label{LNnorm} \|\varphi\|_{\mathbf{V}_{\alpha}^{p,q}} := \left( \int_{\R} \left( \int_{\He^n} \left(\frac{|\varphi(x \cdot (0,0,s)) - \varphi(x)|}{|s|^{\alpha}} \right)^{p} \, dx \right)^{q/p} \, \frac{ds}{|s|} \right)^{1/q}. \end{equation}
\end{definition}

\begin{definition} Let $1 < p < \infty$ and $\alpha > 0$. For $\varphi \in S^p_{\alpha}(\mathbb{H}^n)$, we define
\begin{displaymath} \|\varphi\|_{\dot{S}^{p}_{\alpha}} := \|(-\bigtriangleup_{p})^{\alpha/2}\varphi\|_{p}. \end{displaymath}
\end{definition}
This is a homogeneous counterpart for the $\|\cdot\|_{p,\alpha}$-norm in Definition \ref{d:Sobolev}. With these definitions in place, we have the following inequalities:

\begin{thm}\label{main_LN_section} Let $q \geq 2$, $1 < p \leq q < \infty$, $0 < \alpha < 1$, and $\varphi \in S^{p}_{2\alpha}(\He^{n})$. Then,
\begin{equation}\label{ineq} \|\varphi\|_{\mathbf{V}^{p,q}_{\alpha}} \lesssim_{\alpha,p,q} \|\varphi\|_{V^{p}_{\alpha}}\lesssim_{p} \|\varphi\|_{p,2\alpha}. \end{equation}
Moreover, the following inequality holds for the homogeneous norm $\|\cdot\|_{\dot{S}_{2\alpha}^{p}}$:
\begin{equation}\label{ineq_homog} \|\varphi\|_{\mathbf{V}^{p,q}_{\alpha}} \lesssim_{p,q} \|\varphi\|_{\dot{S}_{2\alpha}^{p}},
\qquad q\geq 2, 1 < p \leq q < \infty, \: 0 < \alpha < 1. \end{equation}
\end{thm}
The second inequality in \eqref{ineq} is nothing but \eqref{form8}, and \eqref{ineq_homog} will follow from \eqref{ineq} and a homogeneity argument. So, essentially the only task is the first inequality in \eqref{ineq}. For this purpose, we introduce two more auxiliary seminorms. Let $1 \leq p,q < \infty$ and $0 < \alpha < 1$. The first new seminorm is the standard \emph{Besov space $\Lambda^{p,q}_{\alpha}(\R)$} seminorm on $\R$ (see \cite[V. \S 5]{MR0290095}):
\begin{displaymath} \|\psi\|_{\Lambda_{\alpha}^{p,q}} := \left( \int_{\R} \left( \int_{\R} \left(\frac{|\psi(t + s) - \psi(t)|}{|s|^{\alpha}} \right)^{p} \, dt \right)^{q/p} \, \frac{ds}{|s|} \right)^{1/q}, \qquad \psi \in \mathcal{S}(\R). \end{displaymath}
The second seminorm is the "vertical Besov seminorm" on $\He^{n}$, defined for $\varphi \in \mathcal{S}(\He^{n})$:
\begin{displaymath} \|\varphi\|_{I(\Lambda_{\alpha}^{p,q})} := \left( \int_{\R^{2n}} \|\varphi_{z}\|_{\Lambda^{p,q}_{\alpha}}^{p} \, dz \right)^{1/p}, \qquad 1 \leq p,q < \infty, \: 0 < \alpha < 1. \end{displaymath}
The seminorms $\|\cdot\|_{I(\Lambda_{\alpha}^{p,q})}$ and $\|\cdot\|_{\mathbf{V}_{\alpha}^{p,q}}$ from \eqref{LNnorm} look quite similar, and there is indeed a simple relation between them:
\begin{proposition}\label{minkowskiProp} Let $\varphi \in \mathcal{S}(\He^{n})$. Then
\begin{displaymath} \|\varphi\|_{\mathbf{V}_{\alpha}^{p,q}} \leq \|\varphi\|_{I(\Lambda_{\alpha}^{p,q})}, \qquad 1 \leq p \leq q < \infty, \: 0<\alpha<1. \end{displaymath}
\end{proposition}
\begin{proof} We write
\begin{displaymath} F(z,s) := \int_{\R} \left(\frac{|\varphi(z,t + s) - \varphi(z,t)|}{|s|^{\alpha}} \right)^{p} \, dt \end{displaymath}
and apply Fubini's theorem and Minkowski's integral inequality:
\begin{align*} \|\varphi\|_{\mathbf{V}_{\alpha}^{p,q}} & = \left(\int_{\R} \left(\int_{\R^{2n}} F(z,s) \, dz \right)^{q/p} \, \frac{ds}{|s|} \right)^{1/q}\\
& \leq \left(\int_{\R^{2n}} \left(\int_{\R} F(z,s)^{q/p} \, \frac{ds}{|s|} \right)^{p/q} \, dz \right)^{1/p}\\
& = \left(\int_{\R^{2n}} \left(\int_{\R} \left(\int_{\R} \left(\frac{|\varphi(z,t + s) - f(z,t)|}{|s|^{\alpha}} \right)^{p} \, dt \right)^{q/p} \, \frac{ds}{|s|} \right)^{p/q} \, dz \right)^{1/p}\\
& = \left(\int_{\R^{2n}} \|\varphi_{z}\|_{\Lambda^{p,q}_{\alpha}}^{p} \, dz \right)^{1/p} = \|\varphi\|_{I(\Lambda_{\alpha}^{p,q})}. \end{align*}
This completes the proof of the proposition. \end{proof}

To proceed, we recall some continuous inclusions between Besov and Sobolev spaces on $\R$, which can be found in Stein's book \cite{MR0290095}. First, Proposition 10 on \cite[p. 153]{MR0290095} yields
\begin{equation}\label{stein1} \|\psi\|_{\Lambda^{p,q}_{\alpha}} \lesssim_{\alpha,p,q} \|\psi\|_{\Lambda_{\alpha}^{p,p}}, \qquad 0 < \alpha < 1, \: 1 \leq p \leq q. \end{equation}
In fact, only the inclusion of Besov spaces $\Lambda_{\alpha}^{p,p}(\R) \subset \Lambda_{\alpha}^{p,q}(\R)$ is mentioned in the statement of \cite[Proposition 10]{MR0290095}, but the proof yields \eqref{stein1}; simply compare \cite[(65), p. 154]{MR0290095} with \cite[(66), p. 154]{MR0290095}. Second, $S^{p}_{\alpha}(\R) \subset \Lambda_{\alpha}^{p,\max\{p,2\}}(\R)$, and indeed
\begin{equation}\label{eq:incl} \|\psi\|_{\Lambda_{\alpha}^{p,\max\{p,2\}}} \lesssim_{\alpha,p} \|\psi\|_{p,\alpha}, \qquad 0 < \alpha < 1, \: 1 < p < \infty. \end{equation}
This is the content of Theorem 5, (A)-(B), on \cite[p. 155]{MR0290095}. Again, the statement only claims the inclusion of spaces, but Stein records \eqref{eq:incl} on the first few lines of \cite[p. 157]{MR0290095}.

Combining Proposition \ref{minkowskiProp} with \eqref{stein1}-\eqref{eq:incl}, and recalling the definition of the vertical Sobolev norm $\|\varphi\|_{V^{p}_{\alpha}}$ from Definition \ref{def:verticalSobolev}, we may infer the first inequality in \eqref{ineq}: for $q \geq 2$, $1 < p \leq q < \infty$, $0 < \alpha < 1$, and $\varphi \in \mathcal{S}(\He^{n})$,
\begin{displaymath} \|\varphi\|_{\mathbf{V}_{\alpha}^{p,q}} \lesssim_{\alpha,p,q} \|\varphi\|_{I(\Lambda_{\alpha}^{p,q})} \stackrel{\eqref{stein1}}{\lesssim_{\alpha,p,q}} \|\varphi\|_{I(\Lambda_{\alpha}^{p,\max\{p,2\}})} \stackrel{\eqref{eq:incl}}{\lesssim_{\alpha,p}} \left( \int_{\R^{2n}} \|\varphi_{z}\|_{p,\alpha}^{p} \, dz \right)^{1/p} =: \|\varphi\|_{V^{p}_{\alpha}}. \end{displaymath}
The density of $\mathcal{S}(\He^{n})$ in $S^{p}_{2\alpha}(\He^{n})$, and the fact that $\|\varphi\|_{V^{p}_{\alpha}}$ is dominated by $\|\varphi\|_{p,2\alpha}$, easily yield the extension of the previous estimate for all $\varphi \in S^{p}_{2\alpha}(\He^{n})$. The proof of Theorem \ref{main_LN_section} is now reduced to verifying the inequality \eqref{ineq_homog}.

\begin{proof}[Proof of \eqref{ineq_homog}] Let $q\geq 2$, $1 < p \leq q < \infty$, and $0 < \alpha < 1$. It suffices to prove \eqref{ineq_homog} for $\varphi \in \mathcal{S}$, by another standard density argument. Fix $\varphi \in \mathcal{S}$. We infer \eqref{ineq_homog} from \eqref{ineq} by the following homogeneity argument. Set $\varphi_r:= \varphi \circ \delta_r$ for $r > 0$, and recall that the Jacobi determinant of $\delta_r$ is $r^Q$ with $Q=2n+2$. Then two applications of the change-of-variables formula  yield
\begin{align*}
\|\varphi_r\|_{\mathbf{V}^{p,q}_{\alpha}}&=  \left( \int_{\R} \left( \int_{\He^n} \left(\frac{|\varphi(\delta_r(x) \cdot (0,0,r^2s)) - \varphi(\delta_r(x))|}{|s|^{\alpha}} \right)^{p} \, dx \right)^{q/p} \, \frac{ds}{|s|} \right)^{1/q}\\
&= r^{-\frac{Q}{p}}  \left( \int_{\R} \left( \int_{\He^n} \left(\frac{|\varphi(y\cdot (0,0,r^2s)) - \varphi(y)|}{|s|^{\alpha}} \right)^{p} \, dy \right)^{q/p} \, \frac{ds}{|s|} \right)^{1/q}\\
&= r^{2\alpha- \frac{Q}{p}}
 \left( \int_{\R} \left( \int_{\He^n} \left(\frac{|\varphi(\delta_r(x) \cdot (0,0,\sigma)) - \varphi(\delta_r(x))|}{|\sigma|^{\alpha}} \right)^{p} \, dx \right)^{q/p} \, \frac{d\sigma}{|\sigma|} \right)^{1/q}
\\&= r^{2\alpha-\frac{Q}{p}} \|\varphi\|_{\mathbf{V}^{p,q}_{\alpha}}.
\end{align*}
On the other hand, by an analogous argument, we find that
\begin{displaymath}
 \|\varphi_r\|_{p,2\alpha} = r^{-\frac{Q}{p}}\|\varphi\|_p + r^{2\alpha-\frac{Q}{p}}  \|\varphi\|_{\dot{S}_{2\alpha}^p},
\end{displaymath}
cf.\  the proof of \cite[Lemma 2.6]{2019arXiv190104767F}. Combining \eqref{ineq} with the previous computations, we get
\begin{displaymath}
 \|\varphi\|_{\mathbf{V}^{p,q}_{\alpha}} \lesssim_{\alpha,p,q} r^{-2\alpha}  \|\varphi\|_p + \|\varphi\|_{\dot{S}_{2\alpha}^p}, \qquad r > 0.
\end{displaymath}
Now \eqref{ineq_homog} follows by letting $r \to \infty$. \end{proof}

\begin{remark} For yet another proof of the original vertical versus horizontal Poincar\'e inequality \eqref{LN}, or at least the cases $1 < p \leq 2 = q$, see \cite[Section 7]{2019arXiv190104767F}. The proof in \cite{2019arXiv190104767F} infers the inequalities from a quantitative affine approximation theorem for horizontal Sobolev functions. \end{remark}

\bibliographystyle{plain}
\bibliography{references}

\end{document}